\documentclass[a4paper,11pt]{amsart}

\usepackage{amssymb}

  \usepackage[colorlinks=true,linkcolor=black,citecolor=black,filecolor=black,urlcolor=black,menucolor=black]{hyperref}
  \usepackage[english]{babel}
  \usepackage{amsmath}
  \usepackage{mathtools}
  \usepackage[utf8]{inputenc}
  \usepackage{xcolor}
  \usepackage{tikz}

\usepackage[
backend=biber,
giveninits=true,
url=false,
maxbibnames=4
]{biblatex}
 \AtEveryBibitem{\clearfield{issn}}
 \AtEveryBibitem{\clearfield{isbn}}
 \AtEveryBibitem{\clearlist{language}} 
 \AtEveryBibitem{\clearfield{month}}
   \renewbibmacro{in:}{}

\usepackage{geometry}
  \theoremstyle{plain}
  \newtheorem{theorem}{Theorem}%[section]
  \newtheorem*{theorem*}{Theorem}
  \newtheorem{proposition}{Proposition}

  \newtheorem{lemma}{Lemma}

  \theoremstyle{definition}
  \newtheorem{definition}{Definition}

  \def \N{\mathbb{N}}
  \def \R{\mathbb{R}}

	\def \S{\mathbb{S}}

  \newcommand {\sign} {\mathop \mathrm{sign}}

  \newcommand {\supp} {\mathop \textup{supp}}
  
 \renewcommand {\div} {\textup{div}}

 \newcommand {\dist} {\textup{dist}}

 \newcommand {\Per} {\textup{Per}}

  \newcommand {\eps} {\varepsilon}
  \newcommand {\Hher} {H}
  \newcommand {\uher} {u}

\DeclareMathOperator*{\argmin}{\mathop \textup{arg\,min}}
\DeclareMathOperator{\med}{\mathop \textup{med}}

\DeclareMathOperator*{\Lip}{Lip}

\DeclareMathOperator*{\erf}{erf}

\DeclareMathAlphabet{\mathbbold}{U}{bbold}{m}{n}

\begin{document}
\author{Fabius Kr\"amer, Tim Laux}

\address{Institut f{\"ur} Mathematik \& Interdisziplin{\"a}res Zentrum f{\"u}r Wissenschaftliches Rechnen, Universit{\"a}t Heidelberg, Im Neuenheimer Feld 205, 69120 Heidelberg, Germany}
\email{\texttt{\{f.kraemer,tim.laux\}@math.uni-heidelberg.de}}

\title[Continuous Median Filter]{Strong Convergence of Continuous Median Filter Scheme}
  \begin{abstract}

    We introduce a continuous-time counterpart of the median filter that gradually denoises a given image. In the limit of vanishing stencil size, our results show that the evolution converges to level-set mean curvature flow. Surprisingly, and for the first time for any median-type filter scheme, we can prove the convergence of energies in this limit. Such strong convergence results are of major interest as they often have to be assumed in the literature to prove convergence of related schemes. Moreover, this result carries crucial geometric information, namely the unit multiplicity of interfaces. The proof relies on a compensated compactness argument inspired by the work of Evans and Spruck.
    \medskip
    

  \medskip

\noindent \textbf{Mathematical Subject Classification (MSC2020)}:
53E10; %Flows related to mean curvature
35D40 (Primary); %Viscosity solutions to PDEs
49Q20;	%Variational problems in a geometric measure-theoretic setting
65M12; %Stability and convergence of numerical methods for initial value and initial-boundary value problems involving PDEs
94A08 (Secondary). % Image processing (compression, reconstruction etc.) in information and communication theory

  \end{abstract}
  \maketitle
  \tableofcontents

\section{Introduction}

Median filters are popular, nonlinear---but computationally efficient---image processing techniques to remove noise while preserving sharp edges.
We propose and analyze a continuous-time counterpart of the median filter that gradually denoises a given image. 
We show that this filter can be viewed as a gradient flow of an associated total variation (TV) energy and therefore efficiently navigates this complex energy landscape.

\medskip

Our main result shows that, in the limit of vanishing stencil size, our continuous-time median filter strongly converges to the unique viscosity solution of level-set mean curvature flow.
A curious feature of this continuous-time median filter is that, in contrast to all other known numerical methods, we can prove that the convergence is \emph{strong} in the sense that the energy converges.
Such strong convergence results are of major interest, as they often have to be assumed in the literature to prove convergence of related schemes, cf.~\cite{LauxOttoBV}.
Moreover, this strong convergence gives crucial geometric insight, namely that the level-sets do not fold over or oscillate. In other words, the multiplicity of interfaces is~$1$.

\medskip

The key tool to prove the strong convergence of our scheme is a uniform-in-time $L^1$ bound on the (non-local) mean curvature, which allows us to carry out a compensated compactness argument inspired by Evans and Spruck~\cite{evans_spruck_IV}.
More precisely, we show that the $L^1$-norm of the non-local mean curvature is non-increasing in time, which is an interesting and desired property for a filter in its own right. In fact, Ambrosio and Masnou~\cite{AmbrosioMasnou}, and Zhu and Chan~\cite{ZhuChan} propose to regularize with this $L^1$-norm of the mean curvature to keep sharp, and relatively smooth edges. 
It is remarkable that our new scheme automatically decreases this functional over time.
    
\medskip

Median filters were first developed for signal processing by Tukey~\cite{Tukey} and generalized to the two-dimensional framework of images by Pratt~\cite{pratt} and Frieden~\cite{Frieden}. 
Huang, Yang, and Tang~\cite{Huang} developed the first computationally efficient implementation of the filter for images. 
Moreover, variants on random point clouds in high dimensions can be used for data clustering and classification tasks~\cite{ullrich_median}.
Median filters can also be viewed as efficient computational methods for level-set mean curvature flow~\cite{OsherSethian}. This was shown by Oberman in his work~\cite{Oberman}, which laid a central foundation for consecutive analytical studies on median filters, including the present one. 
The only continuous time median scheme we could find in the literature is~\cite{zbMATH07200350}. 
However, the analytical part there only focuses on stationary configurations.

\medskip

Median filters are intimately connected to TV regularization, which is a well-established denoising technique, also known as ROF model, introduced by Rudin, Osher, and Fatemi~\cite{RudinOsherFatemi}.
Indeed, Alliney~\cite{alliney92,alliney94} and Li and Osher~\cite{LiOsher} pointed out that median filters are efficient methods for solving certain TV optimization problems.  
Rudin, Osher, and Fatemi instead solve their TV optimization problem via gradient flow-dynamics, now called the TV flow, see~\cite{BellettiniCasellesNovaga} for the analysis.
Our new method combines these two approaches in the sense that our continuous-time median filter is indeed a gradient flow of the TV energy (like the one in~\cite{RudinOsherFatemi}) but has the computational benefit of a median filter (as in~\cite{alliney92,alliney94, LiOsher}).
Our main result proves the convergence of precisely this natural quantity, the TV energy.

\medskip

Recall that a gradient flow is characterized by the choice of energy (here the TV norm) and a Riemannian metric on the space of configurations. 
In case of the standard TV flow~\cite{RudinOsherFatemi}, the metric is the standard $L^2$-scalar product. 
The metric corresponding to our new median filter is related to the geometric (but degenerate) version used for mean curvature flow, in which the $L^2$-scalar product is weighted by the gradient. 
The corresponding limit dynamics, level-set mean curvature flow, were introduced by Osher and Sethian~\cite{OsherSethian}, with first existence and uniqueness theories provided by Evans and Spruck~\cite{evans_spruck_I} and Chen, Giga, and Goto~\cite{ChenGigaGoto}.
Also metrics related to negative norms are of interest, most notably the $H^{-1}$-metric, see the work of Giga et al.~\cite{GigaGiga_2nd4th,GigaKohn_singular,GigaKurodaMatsuoka}.
Another popular metric is the Otto-Wasserstein metric, for which existence is often shown by the JKO scheme due to Jordan, Kinderlehrer, and Otto~\cite{JKO}. 
The existence of weak solutions of the TV flow with respect to the Otto-Wasserstein metric is intricate because of a degeneracy at zero.  Carlier and Poon~\cite{CarlierPoon} provided a conditional result, and the recent work of Lin and Santambrogio~\cite{LinSantambrogio} constructs a lower bound by suitably modifying the JKO scheme.

\medskip

A main motivation for the present work is the revelation by Esedo\u{g}lu, Guo, and Li~\cite{EsedoguleGuoLi} that the discrete-in-time median filter is equivalent to performing the celebrated Merriman-Bence-Osher thresholding algorithm~\cite{MBO92} on each level-set, which simply alternates between linear diffusion and pointwise thresholding. 
Esedo\u{g}lu and Otto~\cite{EsedogluOtto} observed that the MBO scheme can be interpreted as a minimizing movements scheme, which  led to several \emph{conditional} convergence results~\cite{LauxOttoBV,LauxOttoBrakke,LauxOttoDeGiorgi} that all hinge on the assumption of energy convergence.
In the present paper, we prove this convergence for a related, time-continuous version of the median filter. 
It is known  that the MBO scheme itself exhibits this strong convergence under certain geometric conditions, namely as long as the limiting mean curvature flow stays smooth~\cite{SwartzYip} or only has mean convex singularities~\cite{FuchsLaux}. 
These are the only generic singularities in low dimensions, as conjectured by Huisken, shown first by Colding and Minicozzi~\cite{ColdingMinicozzi}, and refined by~Choi, Haslhofer, Hershkovits, and White~\cite{ChoiHaslhoferHershkovitsWhite}. 
Supported by this and the present work, we conjecture that also the MBO scheme converges strongly for generic initial conditions.

\medskip

The key to our strong convergence result is an estimate inspired by Evans and Spruck~\cite{evans_spruck_IV}, which allows us to apply a compensated compactness argument. 
This estimate is a uniform-in-time, $L^1$-in-space estimate on the non-local mean curvature.
In spirit, this estimate is connected to non-local calibrations as introduced by Cabr\'e~\cite{Cabre} for graphs with vanishing non-local mean curvature and Pagliari~\cite{Pagliari} for half-spaces. 
In their static case, these non-local calibrations have vanishing non-local divergence. 
In our time-dependent (and non-graphical) case, one cannot expect this. Even the construction of much simpler \emph{local} time-dependent analogs of calibrations are more complicated as can be seen at the calibrations for the two-dimensional network flow~\cite{FHLS}, and of course their divergence is related to the normal velocity of the moving surfaces. 
Our present result shows that the non-local normal vector fields of the level-sets give rise to an almost-calibration with bounded non-local divergence in the above uniform-in-time and $L^1$-in-space sense. 
We show that this estimate is sufficient to carry out a compensated compactness argument, as these weakly-$\ast$ pre-compact measures are paired with the uniformly converging level-set function.

\section{Main Results}\label{sec:main_result}
Here and in the rest of the paper, we fix $\R^d$ with $d\in \N$ as our domain. 
Let $K\colon \R^d\times\R^d\to\R$ be a given kernel.
The classical median filter is an iterative scheme that, starting from a continuous function $u_0\colon \R^d\to \R$, successively computes 
\begin{equation}\label{eq:med_classical}
    u_n = \med_{K} u_{n-1},
\end{equation}
where the median is defined by 
\begin{equation*}
  \med_K u(x) \in \argmin_{c \in \R} \int K(x,y) |c - u(y)| \,dy.
\end{equation*}
The kernel $K=K_h$ has a characteristic length scale that we call $\sqrt{h}$ (we will later choose a Gaussian of variance $\sim h$), and we can rearrange~\eqref{eq:med_classical} into the more suggestive form
\[
    \frac{u_n-u_{n-1}}{h} = \frac{\med_{K_h} u_{n-1} - u_{n-1}}{h}.
\]
Now, $h$ appears both as a squared length scale (on the right) and as a time scale (on the left). 
We break this symmetry and introduce another time scale $\tau>0$ so that the more general scheme becomes
\begin{equation}\label{eq:med_tau}
    \frac{u_n-u_{n-1}}{\tau} = \frac{\med_{K_h} u_{n-1} - u_{n-1}}{h},
\end{equation}
or equivalently
$
    u_n = \frac{\tau}{h} \med_{K_h} u_{n-1} + (1-\frac{\tau}{h})u_{n-1},
$ 
which is reminiscent of residual neural networks that keep track of information of previous layers (here, $u_{n-1}$). 

We are interested in the case where $\tau\ll h\ll 1$, i.e., we rather do many iterations, then we obtain the continuous-time median filter
\[
    \partial_t u_h = \frac{\med_{K_h} u_h - u_h}{h}.
\]
Drawing again analogies to neural networks, this would be the neural ODE associated to a neural network.

In the rest of the paper we study the time-continuous median filter
\begin{equation}\label{def:semi_discrete_scheme}
\left\{ 
\begin{aligned}
  \partial_t u_h &= \frac{\med_{G_h} u_h - u_h}{h} && \text{in } \R^d\times(0,\infty),\\
  u_h &= u_0 && \text{on } \R^d\times\{t=0\}
  \end{aligned}
\right.
\end{equation}
with $G_h$ a centered Gaussian of variance $2h$, i.e., the heat kernel at time $h$. 
This is a natural choice, as it is the same kernel that appears in the MBO scheme~\cite{MBO92}. 
However, we believe that our analysis applies to a large class of kernels with the only important structural assumption of positivity $K_h\geq0$. Of course, anisotropic kernels will lead to anisotropic motion laws in the limit. We leave this question for future research. In the rest of the present paper we will abbreviate $\med_{h} := \med_{G_h}$ for easier readability.

The median filter dissipates the following energy 
\begin{align}
  E_h(u) :&= \int_{\R^d} \int_{\R^d} G_h(x,y) \Big|\frac{u(x) - u(y)}{\sqrt{h}}\Big| \,dx \, dy. \label{def:E_h}
\end{align}
We prove this in Lemma~\ref{lem:mono_energie} below, but this should be even true for any kernel $K_h$ that is symmetric (i.e., $K_h(x,y)=K_h(y,x)$) and non-negative (i.e., $K_h(x,y)\geq0$), as formally (see Lemma~\ref{lem:mono_energie} for a rigorous proof with $K_h = G_h$) it holds
\begin{multline*}
    \frac{d}{dt} \int_{\R^d} \int_{\R^d} K_h(x,y) \Big|\frac{u_h(x) - u_h(y)}{\sqrt{h}}\Big| \, dy \,dx \\
    =  \int_{\R^d} \frac{1}{\sqrt{h}}\int_{\R^d} K_h(x,y) \sign(u_h(x) - u_h(y)) \,dy \frac{\med_{K_h} u_h(x) - u_h(x)}{h}\, dx \leq 0 .
\end{multline*}
This is a significant generalization of the monotonicity in~\cite{EsedoguleGuoLi}, where non-negativity of the Fourier transform is assumed.

It is easy to see that $E_h$ $\Gamma$-converges to the limit energy $E$ given by 
\[
    E(u)=2c_0 \int_{\R^d} |\nabla u| \, dx\quad \text{with } c_0 = \frac1{\sqrt{\pi}},
\]
for example, by integrating~\cite[Proposition~A.1]{EsedogluOtto} over all level sets.
Due to the monotonicity of the scheme, one can show the convergence to the viscosity solution of level-set mean curvature flow.
\begin{proposition}[Convergence to Viscosity Solution]\label{the:viscosity}
  For initial values $u_0 \in C^3_b(\R^d\setminus S) \cap C^{0,1}(\R^d)$ that are constant outside a bounded set and where $S \subset \R^d$ is closed with $\overline{\R^d\setminus S} = \R^d$, let $u_h$ be the unique classic solution of the median filter~\eqref{def:semi_discrete_scheme} starting from  $u_0$. Then $u_h \rightarrow u$ uniformly as $h \downarrow 0$ and $u$ is the unique viscosity solution of the level-set mean curvature flow equation
  \begin{equation}\label{eq:level_set_equation}
    \partial_t u = \nabla \cdot \left(\frac{\nabla u}{|\nabla u|} \right)|\nabla u|
  \end{equation}
   in the sense of  Definition~\eqref{def:viscosity}.  
\end{proposition}

Our main result states that this convergence is strong in the sense that the energies converge along the sequence. 
Here, the approximate mean curvature 
\begin{equation*}
  H_{0,h}(x) := \frac{2}{\sqrt{h}} \int_{\R^d} G_h(x,y) \sign(u_0(x) - u_0(y)) \,dy \approx 2c_0 \nabla \cdot \Big(\frac{\nabla u_0(x)}{|\nabla u_0(x)|}\Big)
\end{equation*}
associated to our scheme will be a central object.

\begin{theorem}[Strong Convergence]\label{cor:energy_conv}
Let $u_h$ be the solution of~\eqref{def:semi_discrete_scheme} starting from $u_0$, where $u_0$ satisfies the assumption of Proposition~\ref{the:viscosity} and
\begin{equation}\label{def:well_prepared_init}
     \sup_{h\in(0,1]}E_h(u_0)<\infty,
\quad     \sup_{h \in (0,1]} \int_{\R^d} |H_{0,h}(x)| \, dx < \infty.
\end{equation}
 Then it holds, with $u$ the viscosity solution of~\eqref{eq:level_set_equation}, that
  \begin{equation*}
     \lim_{h\to0}E_h(u_h(\cdot,t)) =  E(u(\cdot,t)) \quad \text{for a.e.\ }t>0.
  \end{equation*}
\end{theorem} 

Of course, by the dominated convergence theorem (since $E_h(u_h(\cdot,t)) \leq E_h(u_0)$), this implies
\[
    \lim_{h\to0} \int_0^T E_h(u_h(\cdot,t)) \, dt = \int_0^T  E(u(\cdot,t)) \, dt 
\]
for every $T<\infty$, which is the usual form of the energy convergence assumption, see, e.g., \cite{LauxOttoBV}.

The rest of the paper is organized as follows.
In the next section, we present the heuristic idea of our proof.
In Section~\ref{sec:existence}, we prove well-posedness (existence, uniqueness, and a comparison principle) of the scheme and related versions.
Then, in Section~\ref{sec:convergence}, we prove the convergence to the viscosity solution stated in Proposition~\ref{the:viscosity}.
Section~\ref{sec:conv_energie} is devoted to the proof of our main result, the convergence of the energies.
Some useful properties of median filters are collected in Section~\ref{sec:properties}.
Finally, in Section~\ref{sec:initial}, we construct well-prepared initial conditions.

\section{Idea of Proof}
In this section, we motivate why our main result, the convergence of energies, is true before introducing all the technicalities for the rigorous proof later.

Let us first review the beautiful philosophy of Evans and Spruck~\cite{evans_spruck_IV} on converge of the TV-norm for approximation schemes. So, let $v_h$ be an approximation scheme; it is often not difficult to prove that $v_h$ converges uniformly to some $v$. By standard lower-semicontinuity arguments it then holds
\begin{equation*}
    \liminf_{\eps \downarrow 0} \int_{\R^d} |\nabla v_h| \, dx \geq \int_{\R^d} |\nabla v|\,dx.
\end{equation*}
Thus, one concentrates on proving the upper bound to obtain convergence in the TV-norm.
The key idea is a compensated compactness argument: Simply apply an integration by parts to obtain
\begin{equation}\label{eq:es_motivation}
  \int_{\R^d} |\nabla v_h| \,dx = \left\langle \nabla v_h, \frac{\nabla v_h}{|\nabla v_h|} \right\rangle = -\left\langle v_h, \nabla \cdot \left(  \frac{\nabla v_h}{|\nabla v_h|}\right) \right\rangle.
\end{equation}
We recognize that due to the uniform convergence of $v_h$, one needs the dual, that is, weak-$*$ convergence as measures of the approximate mean curvature $H_h := \nabla \cdot \left(\frac{\nabla v_h}{|\nabla v_h|}\right)$, to pass to the limit in the right-hand side of~\eqref{eq:es_motivation}. To obtain the right compactness, one uses a strange and subtle $L^1$ estimate
\begin{equation*}
    \sup_{t \in [0, \infty)} \sup_{h \in (0,1]} \int_{\R^d} |H_h(x,t)|\,dx < \infty.
\end{equation*}
 The estimate is strange in the sense that it lacks the $|\nabla v|$ term from the metric which would allow decomposing to level-sets by the co-area formula. After establishing the compactness, one can quite quickly establish the right limit. 
But, so far the necessary $L^1$ bound could only be established by Evans and Spruck~\cite{evans_spruck_IV} for their very specific scheme. It is thus rather surprising that we can as well establish such an $L^1$ bound for our approximate mean curvature originating from our application-driven median filter scheme.

\medskip

In order to follow the Evans--Spruck philosophy, we briefly introduce non-local calculus:
  The non-local gradient of a function $f:\R^d \to \R$ is defined by
  \begin{equation*}
    \nabla_h f(x,y) := \frac{f(x) - f(y)}{\sqrt{h}}.
  \end{equation*}
  For a non-local vector field $\zeta: \R^d \times \R^d \to \R$ we define the non-local divergence as the adjoint of the non-local gradient with respect to the (by $G_h$ weighted) $L^2$-scalar product, i.e, 
  \begin{equation*}
    \int_{\R^d} \int_{\R^d} G_h(x,y) \zeta(x,y) \nabla_h u(x,y) \,dx \,dy =: \langle \zeta, \nabla_h u\rangle_h \stackrel{!}{=} \langle \div_h \zeta, u \rangle := \int_{\R^d} \div_h \zeta(x) u(x) \,dx.
  \end{equation*}
  One checks that this definition implies for any non-local vectorfield $\zeta$ that
  \begin{equation*}
      \div_h \zeta (x) := \frac{1}{\sqrt{h}}\int_{\R^d} G_h(x,y) (\zeta(x,y)-\zeta(y,x)) \, dy.
  \end{equation*}

The integration by parts mirroring~\eqref{eq:es_motivation} for our non-local TV-norm reads 
\begin{align}\label{eq:discrete_IBP_example}
  E_h(u_h) = \int_{\R^d} \int_{\R^d} G_h(x,y) |\nabla_h u_h(x,y)|\, dx \, dy = \langle \nabla_h u_h, \sign( \nabla_h u_h) \rangle_h = \langle u_h, \div_h \big(\sign( \nabla_h u_h) \big)\rangle.
\end{align}
Thus, in order to prove convergence it will be sufficient to prove that the approximate mean curvature 
\begin{equation*}
  H_h(x) :=  \div_h \big(\sign(\nabla_h u_h) \big)(x) = \frac{2}{\sqrt{h}} \int_{\R^d} G_h(x,y) \sign(u_h(x) - u_h(y)) \,dy 
\end{equation*}
converges weakly-$*$ as a measure to some well-behaved quantity. Indeed, by Proposition~\ref{the:viscosity}, $u_h$ converges uniformly to the unique viscosity solution; so weak-$*$ convergence of $H_h$ as dual to uniform convergence is sufficient to pass to the limit in the product. Therefore, the main challenge is exactly to show the crucial $L^1$ bound
\begin{equation}\label{eq:formal_goal}
  \sup_{h \in (0,1]} \sup_{t \in [0,T]} \int_{\R^d} |H_h| \,dx <\infty, 
\end{equation}
which implies the necessary compactness via Banach-Alao\u{g}lu. Assume for a moment that~\eqref{eq:formal_goal} is true, then, up to subsequences, $H_h \xrightharpoonup{*} \mu$ weakly-$*$ as measure for some measure $\mu$ that can be shown to be of the form $\mu = \nabla \cdot \sigma dx \, dt$ with some vector field $\sigma$ satisfying $|\sigma|\leq 2 c_0 = \frac{2}{\sqrt{\pi}}$. This yields together with the uniform convergence of $u_h$ that
\begin{align*}
  \limsup_{h \downarrow 0} \int_0^T E_h(u_h) \,dt &= \limsup_{h \downarrow 0} \int_0^T \langle u_h, \div_h \big(\sign(\sqrt{h}\nabla_h u_h) \big)\rangle \,dt\\
  &= \int_0^T \int_{\R^d} u(x,t) \nabla \cdot \sigma(x,t) \, dx \,dt \\
  &= - \int_0^T \int_{\R^d} \nabla u(x,t) \cdot \sigma(x,t) \, dx \,dt \\
  &\leq 2c_0 \int_0^T \int_{\R^d} |\nabla u(x,t)|\,dx \,dt. 
\end{align*}
This is precisely the important upper bound for the energy convergence; the lower bound is a standard lower-semicontinuity argument.

So, what is the miraculous mechanism behind~\eqref{eq:formal_goal}? As a motivation we make the following formal calculation:
\begin{align*}
  \frac{d}{dt} \int_{\R^d} |H_h(x,t)|\,dt &= \int \sign(H_h(x,t)) \partial_t H_h(x,t) \,dx \\
  &= \int_{\R^d} \sign(H_h(x,t)) \div_h \big( 2\delta_{\nabla_h u_h = 0} \nabla_h \partial_t u_h \big) \,dx \\
  &= \int_{\R^d} \sqrt{h}\, \div_h \big( 2\nabla_h \sign(H_h(\cdot,t)) \ \delta_{\nabla_h u_h = 0} \big) \partial_t u_h \,dx.
\end{align*}
We now claim that 
\begin{equation}\label{eq:claim_different_sign}
\div_h \big( \nabla_h \sign(H_h(\cdot,t)) \ \delta_{\nabla_h u_h = 0} \big)(x)\partial_t u_h(x) \leq 0
\end{equation}
 for any $x \in \R^d$. This claim implies
 \begin{equation*}
  \frac{d}{dt} \int_{\R^d} |H_h(x,t)|\,dt \leq 0
 \end{equation*}
 which itself yields~\eqref{eq:formal_goal} for well-prepared initial conditions. 

To see that the claim~\eqref{eq:claim_different_sign} is true, fix $x \in \R^d$. If $ \partial_t u_h(x) = 0$ we are immediately done. So, assume (the other case is similar) that 
\begin{equation*}
  \frac{\med_h u_h(x) - u_h(x)}{h} = \partial_t u_h(x) > 0.
\end{equation*}
It has to hold $\sign(H_h(x)) \leq 0$ as otherwise one would have the contradiction (using the monotonicity of the $\sign$-function with $\med_h u_h(x) > u_h(x)$)
\begin{align}
  0 < H_h(x) &= \frac{2}{\sqrt{h}} \int_{\R^d} G_h(x,y) \sign(u_h(x) - u_h(y)) \,dy \nonumber\\
  &\leq \frac{2}{\sqrt{h}} \int_{\R^d} G_h(x,y) \sign(\med_h u_h(x) - u_h(y)) \,dy =0, \label{eq:sign_trick}
\end{align} 
where we used in the last equation the Euler-Lagrange equation associated to the definition of the median and assuming that the level-set of $\med_h u_h(x)$ has $0$ measure.
As $\sign \in \{-1,0,1\}$ it follows first $\sign(H_h(x)) = -1$ and thus also $\nabla_h \sign(H_h)(x,y) \leq 0$. As $G_h(x,y) \geq 0$ and $\delta_{\nabla_h u_h = 0} \geq 0$ we conclude
\begin{equation*}
  \div_h \big( \nabla_h \sign(H_h) \ \delta_{\nabla_h u_h = 0} \big)(x) = \frac{2}{\sqrt{h}}\int_{\R^d} G_h(x,y) \nabla_h \sign(H_h) (x,y) \delta_{\nabla_h u_h(x,y) = 0} \,dy \leq 0,
\end{equation*}
which is precisely the claim~\eqref{eq:claim_different_sign}. We will later make this formal argument rigorous by several regularization procedures.

\section{Existence and Uniqueness of the Scheme}
\label{sec:existence}
In this section we lay the foundation of the paper; that is, we show there exists a unique solution to the scheme~\eqref{def:semi_discrete_scheme}. To that aim we prove a comparison principle that will be useful throughout this paper. At the end we will also prove that the $L^\infty$ norm of the time derivative of the scheme is non-increasing which will be later crucial to verify that the initial conditions are attained in the limit $h \downarrow 0$.
As already mentioned in the previous section, we will later work with more regularized versions of our median filter scheme~\eqref{def:app_scheme}. Thus, we introduce the following general framework:
 Let $K$ be an arbitrary kernel satisfying
  \begin{gather}
    K: \R^d \to \R_{\geq 0}, \quad \quad \int_{\R^d} K(z) \,d z < \infty, \nonumber  \\
    \text{ and either } \label{def:K_kernel}\\     
    \supp K \text{ is bounded} \quad \quad \text{ or } \quad \inf_{z \in B_R(0)} K(z) > 0 \text{ for any } R > 0,\nonumber
  \end{gather}
  as well as $\Psi: \R \rightarrow \R$ convex, differentiable except at the point $0$ and
  \begin{equation*}
    \Psi' =: \psi \text{ in } \R \setminus \{0\}, \quad \quad \psi(s) = - \psi(-s), \quad \quad \psi \neq 0 \text{ on } \R\setminus\{0\}.
  \end{equation*}
  Then we define the generalized median operator 
  \begin{equation}\label{def:generalized_median}
      \med_{K, \Psi} f(x) \in \argmin_{c \in \R} \int_{\R^d} K(x-y) \Psi\big(c - f(y) \big)\, dy
  \end{equation}
and the generalized median filter scheme by
\begin{equation}\label{def:generlized_scheme}
  \begin{cases}
    \partial_t w = \frac{\med_{K,\Psi} w - w}{h} & \text{ in } \R^d \times (0, \infty),\\
    w(\cdot, 0) = w_0 & \text{on } \R^d \times \{t = 0\}.
  \end{cases}
\end{equation}

We start by proving a comparison principle for this generalized scheme which immediately implies the uniqueness of solutions to the scheme~\eqref{def:generlized_scheme}.
\begin{lemma}[General Comparison Principle]\label{lem:comp_princ}
    Assume that $\underline{w}$ and $\overline{w}$ are a sub- and super-solution to~\eqref{def:generalized_median}, in the sense that
    \begin{equation*}
      \partial_t \underline{w} \leq \frac{\med_{K, \Psi} \underline{w} - \underline{w}}{h} \quad \quad \text{ and }\quad \quad\partial_t \overline{w} \geq \frac{\med_{K, \Psi} \overline{w} - \overline{w}}{h},
    \end{equation*}
    respectively, with bounded continuous initial data satisfying  $\underline{w}(\cdot, 0) \leq \overline{w}(\cdot, 0)$, and such that $\lim_{|x|\to \infty} \underline{w}(x,t) \leq \lim_{|x|\to \infty} \overline{w}(x,t)$ exist for all $t \geq 0$. Then it holds for all $t \geq 0$ that %\textcolor{red}{or is it enough to assume $\limsup_{|x|\to \infty} \underline{w}(x)-\overline{w}(x) \leq 0?$}
    \begin{equation*}
        \underline{w}(\cdot, t) \leq \overline{w}(\cdot, t).
    \end{equation*}
\end{lemma}
\begin{proof}
    \emph{Step 1 -- Strict Inequalities:} First, we prove the claim if it holds
    \begin{equation*} 
      \underline{w}(\cdot, 0) < \overline{w}(\cdot, 0) \quad \quad \text{ and } \quad\quad \lim_{|x|\to \infty} \underline{w}(x,t) < \lim_{|x|\to \infty} \overline{w}(x,t)
    \end{equation*}
     for all $t \geq 0$, and if $\overline{w}$ is a strict super-solution, i.e, 
    \begin{equation*}
      \partial_t \overline{w} > \frac{\med_{K, \Psi} \overline{w} - \overline{w}}{h}.
    \end{equation*}

    We argue by contradiction: Assume that $\underline{w} \geq \overline{w}$ at some point.  This, together with the continuity of $\underline{w}$ and $\overline{w}$ on $\R^d \times (0, \infty)$ implies that we can pick the smallest $t_0 > 0$ such that there exists an $x_0 \in \R^d$ with $\underline{w}(x_0,t_0) = \overline{w}(x_0,t_0)$. By construction, it also holds $\underline{w}(x,t_0) \leq \overline{w}(x,t_0)$ for all $x \in \R^d$. As $\underline{w}(x_0, t_0 - \eps) - \overline{w}(x_0, t_0 -\eps ) \leq 0 = \underline{w}(x_0, t_0 ) - \overline{w}(x_0, t_0 )$ for small $\eps$ it holds 
    \begin{align*}
        0 \leq \partial_t (\underline{w}(x_0, t_0) - \overline{w}(x_0, t_0)) h &< \med_{K, \Psi} \underline{w}(x_0,t_0) - \med_{K, \Psi} \overline{w}(x_0,t_0) + \overline{w}(x_0, t_0) - \underline{w}(x_0, t_0) \\ 
        &= \med_{K, \Psi} \underline{w}(x_0,t_0) - \med_{K, \Psi}\overline{w}(x_0,t_0).
    \end{align*}
    But this is a contradiction, because the monotonicity of the median operator by Lemma~\ref{lem:uniqueness_median}, together with $\underline{w}(x,t_0) \leq \overline{w}(x,t_0)$ for all $x \in \R^d$, implies 
    \begin{equation*}
        \med_{K, \Psi} \underline{w}(x_0,t_0) - \med_{K, \Psi} \overline{w}(x_0,t_0) \leq 0.
    \end{equation*}
    
    \emph{Step 2 -- Conclusion:} 
    Define $\tilde{w}(\cdot, t) := \overline{w}(\cdot, t) + \delta (t +1)$ for any $t \geq 0$ and some $\delta > 0$.
    Then $\tilde{w}$ is a strict super-solution as
    \begin{align*}
      \partial_t \tilde{w}(\cdot, t) = \partial_t \overline{w} + \delta &> \frac{\med_{K, \Psi} \overline{w}(\cdot, t) - \overline{w}(\cdot, t)}{h} \\ &= \frac{\med_{K, \Psi} (\overline{w}(\cdot, t) + \delta (t + 1)) - (\overline{w}(\cdot, t) + \delta (t +1))}{h} = \frac{\med_{K, \Psi} \tilde{w}(\cdot, t) - \tilde{w}(\cdot,t)}{h},
    \end{align*}
        where we used that for any $c \in \R$ and any continuous function $f$ it holds $\med_{K,\psi}(f + c) = \med_{K, \Psi}(f) + c$.
    Additionally, one has $\underline{w}(\cdot, 0) < \tilde{w}(\cdot, 0)$ and $\lim_{|x|\to \infty} \underline{w}(x,t) < \lim_{|x|\to \infty} \tilde{w}(x,t)$. This brings us in the position to apply \emph{Step 1} to the pair $\underline{w}, \tilde{w}$ yielding
    \begin{equation*}
      \underline{w}(x, t) \leq \tilde{w}(x, t) = \overline{w}(x, t) + \delta t \quad \quad \text{ for any } (x,t) \in \R^d \times [0, \infty).
    \end{equation*}
    As $\delta \geq 0$ was arbitrary, we can pass to the limit $\delta \to 0$ which finishes this proof.
\end{proof}
\begin{proposition}[Existence]
    For any initial value $w_0\in C^{0,1}(\R^d)$ that is constant outside a bounded set there exists a classic solution to~\eqref{def:generlized_scheme}.
\end{proposition}
\begin{proof}
    We want to approximate the equation~\eqref{def:generlized_scheme} by replacing the time derivative with a difference quotient. To that aim define iteratively
    \begin{equation}\label{def:app_scheme}
        w_\tau^{n+1} = \frac{\tau}{h}(\med_{K, \Psi} w_\tau^n -w_\tau^n) + w_\tau^n
    \end{equation}
    with the initial condition $w_\tau^0 = w_0$ as given in~\eqref{def:semi_discrete_scheme}.
    Rearranging immediately yields the implicit Euler formula
    \begin{equation}\label{eq:discrete_equation}
        \frac{w_\tau^{n+1} - w_\tau^n}{\tau} = \frac{\med_{K, \Psi} w_\tau^n - w_\tau^n}{h}.
    \end{equation}
    \emph{Step 1 -- Monotonicity in $L^\infty$: } We first claim that if $\tau < h$, then it holds
    \begin{equation}\label{eq:mon_Linfty_app_scheme}
        \sup_{x \in \R^d} |w_\tau^{n+1}| \leq \sup_{x \in \R^d} |w_\tau^n|,
    \end{equation}
    which implies that $w_\tau^n$ is uniformly bounded.

    Indeed, as it holds $0 \leq 1- \frac{h}{\tau} \leq 1$ and $|\med_{K, \Psi} v(x_0)| \leq \sup_{y \in R^d} |v(y)|$ for any function $v$ by Lemma~\ref{lem:uniqueness_median} (comparing to a constant function) and $x_0 \in \R^d$ we simply compute
    \begin{align*}
        |w_\tau^{n+1}(x_0)| &= |\frac{\tau}{h} \med_{K, \Psi} w_\tau^n(x_0) + (1- \frac{h}{\tau} ) w_\tau^n(x_0)|\\
        & \leq \frac{\tau}{h}| \med_{K, \Psi} w_\tau^n(x_0)| + (1- \frac{h}{\tau} ) |w_\tau^n(x_0)| \\
        & \leq \frac{\tau}{h} \sup_{x \in \R^d} |w_\tau^n(x)| + (1- \frac{h}{\tau} ) \sup_{x \in R^d}|w_\tau^n(x)|\\
        &\leq \sup_{x \in \R^d} |w_\tau^n(x)|.
    \end{align*}
    As $x_0$ was arbitrary, the claim follows.
    
    \emph{Step 2 -- Linear Interpolation:} We introduce the linear interpolation $w_\tau: \R^d \times (0, \infty) \rightarrow \R$, i.e., $w_\tau$ is piecewise affine in time and it holds 
    \begin{equation*}
      w_\tau(x, n \tau) = w_\tau^n(x) \quad \text{ for all } n \in \N. 
    \end{equation*}

    \emph{Step 3 -- Equicontinuity in Space and Time:} We claim that the family of functions $\{w_\tau\}_\tau$ of the linear interpolations is equicontinuous. 

    We start with the equicontinuity in time: Notice the time-slope of the linear interpolation is bounded as it holds by \emph{Step 1} that
    \begin{align*}
      \left| \frac{w_\tau( \cdot, (n+1) \tau) - w_\tau( \cdot, n\tau)}{\tau} \right| = \left| \frac{\med_{K, \Psi} w_\tau^n(\cdot) - w_\tau^n(\cdot)}{h} \right| \leq \frac{2  \|w_0\|_\infty}{h}.
    \end{align*}
    As $h$ is fixed this immediately implies the claimed equicontinuity in time.

    We continue with the equicontinuity in space:
   Later, in Lemma~\ref{lem:contraction} we will prove that the generalized median operator is 1-Lipschitz, i.e.,
    \begin{equation}\label{eq:med_contraction}
      |\med_{K, \Psi} f(x) - \med_{K, \Psi} f(y)| \leq L(f) |x-y| \quad \text{ for all } x,y \in \R^d,
    \end{equation}
    for any continuous, bounded $f$ and the Lipschitz constant $L(f)$.
    
    Equipped with~\eqref{eq:med_contraction} we can now compute that the Lipschitz constant of $w_\tau$ is also non-increasing, i.e.,
    \begin{align*}
      |w_\tau^{n+1}(x) - w_\tau^{n+1}(y)| &\leq \frac{\tau}{h} |\med_{K, \Psi} w_\tau^n (x) - \med_{K, \Psi} w_\tau^n(y)| + (1- \frac{\tau}{h}) |w_\tau^n(x) - w_\tau^n(y)| \\
      &\leq \frac{\tau}{h} L(w_\tau^n)|x-y| + (1- \frac{\tau}{h}) L(w_\tau^n) |x-y|\\
      &= L(w_\tau^n) |x-y|.
    \end{align*}
    By induction, it follows
    \begin{equation} \label{eq:bound_lipschitz}
      L(w_\tau^{n+1}) \leq L(w_\tau^{n}) \leq \dots \leq L(w_0),
    \end{equation}
    which implies the space-equicontinuity of $w_\tau$.

    \emph{Step 4 -- Compactness of $w_\tau$:} The requirements of the Arzel{\`a}-Ascoli theorem are fulfilled by \emph{Step~1} and \emph{Step~3}. Hence, there exists a $w$ such that $w_\tau \rightarrow w$ uniformly on $B_R(0) \times (0, T)$ for any $R, T > 0$. Furthermore, $w$ is Lipschitz in space with 
    \begin{equation}\label{eq:Lip_u_to_u_0}
      \Lip(w) \leq \Lip(w_0).
    \end{equation}

    \emph{Step 5 -- Equicontinuity of Difference Quotient:}
    We show that the difference quotient 
    \begin{equation}\label{def:difference_quotient}
      \frac{w_\tau(x, t + \tau) - w_\tau(x, t)}{\tau} 
    \end{equation}
    is equicontinuous in $x \in \R^d$ and $t \in [0, \infty)$. 

    We start with the time-equicontinuity: By \emph{Step 3} it holds for any $s,t \geq 0$ and $z \in \R^d$ that $|w_\tau(z, t) - w_\tau(z, s)| \leq \frac{2\|w_0\|_\infty}{h} |t-s|$. Again, as the median operator is $1$-Lipschitz by Lemma~\ref{lem:contraction} that implies
    \begin{equation*}
      |\med_{K, \Psi} w_\tau (x, t) - \med_{K, \Psi} w_\tau(x, s)| \leq \sup_{z \in \R^d} |w_\tau(z, t) - w_\tau(z, s)| \leq \frac{2\|w_0\|_\infty}{h} |t-s|.
    \end{equation*} 
    By our definition of the scheme~\eqref{def:app_scheme} it follows now for $t = n \tau$ and $n \in \N$
    \begin{align*}
      &\left|\frac{w_\tau(x, t + \tau) - w_\tau(x, t)}{\tau} - \frac{w_\tau(x, s + \tau) - w_\tau(x, s)}{\tau} \right| \\
      &= \left|\frac{\med_{K, \Psi} w_\tau(x,t) - \med_{K, \Psi} w_\tau(x,s)}{h} + \frac{w_\tau(x, s) - w_\tau(x,t)}{h} \right|\\
      &\leq \frac{4\|w_0\|_\infty}{h^2} |t-s|.
    \end{align*}
    As the difference quotient is linear in $w_\tau$ the result follows by the properties of the linear interpolation.
    For the space-equicontinuity we use again~\eqref{eq:med_contraction} and \eqref{eq:bound_lipschitz} of \emph{Step 3} to compute for $t = n \tau$ and $n \in \N$
    \begin{align*}
      &\left|\frac{w_\tau(x, t + \tau) - w_\tau(x, t)}{\tau} - \frac{w_\tau(y, t + \tau) - w_\tau(y, t)}{\tau} \right| \\
      &= \left|\frac{\med_{K, \Psi} w_\tau(x, t) - \med_{K, \Psi} w_\tau(y, t)}{h} + \frac{w_\tau(y, t) - w_\tau(x, t)}{h} \right|\\
      &\leq \frac{2 L(w_0)}{h}|x-y|.
    \end{align*}
    Again, the result for general $t$ follows by the properties of the linear interpolation, i.e., for any Lipschitz continuous $f^0, f^1$ and the linear interpolation $f$ between them it holds
    \begin{align*}
      |f(x,t) - f(y,t)| &= |(1-t) f^0(x ) + t f^1(x) - (1-t) f^0(y) - t f^1(y)| \\
      &\leq (1-t) |f^0(x) - f^0(y)| + t|f^1(x) - f^1(y)| \\
      &\leq \max\{L(f^0), L (f^1)\} |x - y|.
    \end{align*}

    \emph{Step 6 –- Compactness of the Difference Quotient:} We claim that the difference quotient converges to the time derivative of $w$, the limit of \emph{Step 4}.

    The equicontinuity in time of $w$ form \emph{Step~3} gives the equiboundedness of the difference quotients~\eqref{def:difference_quotient} and \emph{Step~5} gives the equicontinuity. By the Arzel{\`a}-Ascoli theorem (up to subsequences)
    \begin{equation}\label{eq:lim_differnce_quotient}
        \frac{w_\tau(x, t + \tau) - w_\tau(x, t)}{\tau} \rightarrow f(x,t)
    \end{equation}
    uniformly as $\tau \downarrow 0$ for any $x \in B_R(0)$ and $t \in [0,T]$. It is also clear that $f$ is Lipschitz-continuous.
    By the linear interpolation in time it holds
    \begin{equation*}
      \frac{w_\tau(x, t + \tau) - w_\tau(x, t)}{\tau}  = \partial_t w_\tau(x,t) \quad \quad \text{ for } t \in [0, \infty) \setminus \N\tau.
    \end{equation*}
    This lets us conclude, by the fundamental theorem of calculus and the continuity of $w_\tau$ in time, that for any $0 \leq s_1 \leq s_2 \leq T$
    \begin{equation*}
      \int_{s_1}^{s_2} \frac{w_\tau(x, t + \tau) - w_\tau(x, t)}{\tau}\, dt = w_\tau(s_2) - w_\tau(s_1).
    \end{equation*}
    Now we can use~\eqref{eq:lim_differnce_quotient} to pass to the limit $\tau \downarrow 0$ to obtain
    \begin{equation*}
      \frac{1}{s_2-s_1}\int_{s_1}^{s_2} f(t) \, dt = \frac{w(s_2) - w(s_1)}{s_2 - s_1}.
    \end{equation*}
    As $f$ is continuous we can pass $s_1 \rightarrow s_2$ to obtain
    \begin{equation*}
      f(t) = \partial_t w(t)
    \end{equation*}
    for any $t \geq 0$. Hence, $w$ is differentiable in time and it holds
    \begin{equation*}
        \frac{w_\tau(t + \tau, x) - w_\tau(t, x)}{\tau} \rightarrow \partial_t w(t,x).
    \end{equation*}

\emph{Step 7 -- Convergence of the Median:} We claim, for any $x \in \R^d$, $t \geq 0$ it holds
\begin{equation}\label{eq:conv_med_tau}
  \lim_{\tau \downarrow 0} \med_{K, \Psi} u_{\tau}(x,t) = \med_{K, \Psi} w(x,t),
\end{equation}
where $w$ is the limit of $\emph{Step 4}$.

To prove this, fix $x \in \R^d$, $t\geq 0$ and $\eps > 0$. By Lemma~\ref{lem:cuttable} combined with $\max\{\Lip(w_\tau), \Lip(w)\} \leq \Lip(w_0)$ there exists an $R > 0$ such that for the cut-off kernel $K_R(x,y) = \mathbbold{1}_{|x-y|\leq R} K(x,y)$ it holds
\begin{equation*}
  |\med_{K_R, \Psi} w_\tau(x,t) - \med_{K, \Psi} w_\tau(x,t)| \leq \frac{\eps}{3} \quad \quad \text{ and } \quad \quad |\med_{K_R, \Psi} w(x,t) - \med_{K, \Psi} w(x,t)| \leq \frac{\eps}{3}.
\end{equation*}
Furthermore, by the local uniform convergence of \emph{Step 4} there exists a $\tau_0$ such that for $\tau < \tau_0$ it holds 
\begin{equation*}
\|w_\tau(\cdot,t) - w(\cdot, t)\|_{L^\infty(B_R(0))} \leq \frac{\eps}{3}.
\end{equation*}
This yields together with Lemma~\ref{lem:contraction} the bound
\begin{equation*}
   |\med_{K_R, \Psi} w_\tau(x,t) - \med_{K_R, \Psi} w(x,t)| \leq \|w_\tau(\cdot, t) - w(\cdot, t)\|_{L^\infty(\supp K_R(x, \cdot))} \leq \frac{\eps}{3}.
\end{equation*}
To summarize, for $\tau \leq \tau_0$ it holds
\begin{align*}
&\left|\med_{K,\Psi} w_\tau(x,t)-\med_{K,\Psi} w(x,t)\right| \\
&\leq \left|\med_{K,\Psi} w_\tau(x,t)-\med_{K_R,\Psi} w_\tau(x,t)\right|
\,+\, \left|\med_{K_R,\Psi} w_\tau(x,t)-\med_{K_R,\Psi} w(x,t)\right| \\
&\qquad + \left|\med_{K_R,\Psi} w(x,t)-\med_{K,\Psi} w(x,t)\right| \\
&\leq \varepsilon.
\end{align*}
As $\eps$ was, arbitrary this finishes \emph{Step 7}.

    \emph{Step 8 -- Conclusion:} We claim we can pass to the limit in~\eqref{eq:discrete_equation}.

    By the local uniform convergence of $w_\tau$ from \emph{Step 4}, the local uniform convergence of the difference quotient from \emph{Step 6}, and the pointwise convergence of the median from \emph{Step 7}, we can indeed pass to the pointwise limit $\tau \to 0$ in~\eqref{eq:discrete_equation} to obtain
    \begin{equation*}
      \partial_t w(x,t) = \frac{\med_{K, \Psi} w(x,t) - w(x,t) }{h}.
    \end{equation*}
    Thus, $w$ is a classical solution to \eqref{def:generlized_scheme}.
\end{proof}

We finish this section by observing that the $L^\infty$ norm of the time derivative of solutions of~\eqref{def:generlized_scheme} is non-increasing. This will be used to check that the limit indeed attains the initial conditions.
\begin{lemma}\label{lem:mon_time_derivative}
  Let $w$ be a solution of~\eqref{def:generlized_scheme} with initial value $w_0\in C_b^{0,1}(\R^d)$ then it holds for $0 \leq s \leq t < \infty$
  \begin{equation*}
    \|\partial_t w(\cdot,t)\|_{\infty} \leq \|\partial_t w(\cdot, s)\|_{\infty}.
  \end{equation*}
\end{lemma}
\begin{proof}
  Let $w$ be a solution of~\eqref{def:generlized_scheme} and $\tau > 0$ be arbitrary, then it holds
  \begin{equation*}
    w(x,t + \tau) = w(x,t) + \tau \partial_t w(x,t) + \int_{t}^{t + \tau} \partial_t w(x,s) - \partial_t w(x,t) \, ds.
  \end{equation*}
  We use the contraction property of Lemma~\ref{lem:contraction} to compute
  \begin{multline}\label{eq:partial_t_contraction}
    |\partial_t w(x,s) - \partial_t w(x,t)| \leq \left| \frac{\med_{K, \Psi} w(x,s) - \med_{K, \Psi} w(x,t)}{h} \right| + \left| \frac{w(x,s) - w(x,t)}{h} \right| \\
    \leq \frac{2}{h} \|w(\cdot,s) - w(\cdot,t)\|_{C^0(\R^d)}.
  \end{multline}
  Furthermore, we crudely estimate, using $|\med_{K, \Psi} w(\cdot,t)| \leq \|w(\cdot, t)\|_\infty \leq \|w_0\|_\infty$ by~\eqref{eq:mon_Linfty_app_scheme},
  \begin{equation}\label{eq:partial_t_contraction2}
    \|w(\cdot,s) - w(\cdot,t)\|_{C^0(\R^d)} \leq |s-t|\|\partial_t w\|_{C^0(\R^d \times [0,\infty))} \leq |s-t| \frac{2\|w_0\|_{\infty}}{h}.
  \end{equation}
  Combining the above computations yields
  \begin{equation*}
    w(t+\tau) = w(t) + \tau \partial_t w(t) + \mathcal{O}\big(\frac{\tau^2}{h^2}\big).
  \end{equation*}
  So for two solutions $v, w$ of~\eqref{def:generlized_scheme} we have again by Lemma~\ref{lem:contraction}
  \begin{align}\label{eq:difference_quotient_second_order}
    &\|w(\cdot, t+\tau) - v(\cdot, t+\tau)\|_\infty \\
    &\leq (1-\frac{\tau}{h})\|w(\cdot, t) - v(\cdot, t)\|_\infty + \frac{\tau}{h}\|\med_{K, \Psi} w(\cdot, t) - \med_{K, \Psi} v(\cdot, t)\|_\infty + \mathcal{O}\big(\frac{\tau^2}{h^2}\big) \nonumber\\
    &\leq \|w(\cdot, t) - v(\cdot, t)\|_\infty + \mathcal{O}\big(\frac{\tau^2}{h^2}\big) \nonumber.
  \end{align}
  Denote $\zeta(t):= \|w(\cdot, t) - v(\cdot, t)\|_\infty$ and compute for $0 \leq s \leq t < \infty$ that
  \begin{align*}
    \zeta(t) - \zeta(s) = \sum_{i = 1}^N \zeta\Big(t - \frac{i-1}{N}(t-s)\Big) - \zeta\Big(t - \frac{i}{N}(t-s)\Big) \leq \sum_{i = 1}^N \mathcal{O}\Big(\frac{|s-t|^2 }{(Nh)^2}\Big) = \mathcal{O}\Big(\frac{|s-t|^2 }{N h^2}\Big).
  \end{align*}
  As $s,t$ and $h$ are fixed we can pass $N \to \infty$ to obtain
  \begin{equation*}
    \|w(\cdot, t) - v(\cdot, t)\|_\infty \leq \|w(\cdot, s) - v(\cdot, s)\|_\infty.
  \end{equation*}
  Now using $v(\cdot, t) := w(\cdot, t + \tau)$ we obtain
  \begin{equation}\label{eq:before_pass_tau}
    \frac{1}{\tau}\|w(\cdot, t) - w(\cdot, t + \tau)\|_\infty \leq \frac{1}{\tau}\|w(\cdot, 0) - w(\cdot, \tau)\|_\infty.
  \end{equation}
  We can pass to the limit in the above estimate if we can show that for arbitrary $t \geq 0$ 
  \begin{equation*}
    \lim_{\tau \to 0} \frac{1}{\tau}\|w(\cdot, t) - w(\cdot, t + \tau)\| _\infty = \|\partial_t w(\cdot, t)\|_\infty.
  \end{equation*}
  But this follows by the inverse triangle inequality and the above computations~\eqref{eq:partial_t_contraction} and \eqref{eq:partial_t_contraction2}:
  \begin{align*}
    \left|\frac{1}{\tau}\|w(\cdot, t + \tau) - w(\cdot, t)\|_\infty - \|\partial_t w(\cdot, t)\|_\infty \right| &\leq \left\|\frac{w(\cdot, t + \tau) - w(\cdot, t)}{\tau} - \partial_t w(\cdot, t) \right\|_\infty\\
    &= \left\| \frac{1}{\tau} \int_{t}^{t + \tau} \partial_t w(\cdot, s) - \partial_t w(\cdot, t)\,ds\right\|_\infty\\
    &\lesssim \frac{1}{\tau} \int_{t}^{t + \tau} \frac{|s-t|}{h^2} \,ds\\
    & = \frac{\tau}{h^2}.
  \end{align*}
  Thus, we can pass to the limit in~\eqref{eq:before_pass_tau} and finish the proof.
\end{proof}

\section{Convergence to the Viscosity Solution}
\label{sec:convergence}

The level-set mean curvature flow equation is given by
\begin{equation*}
    \frac{\partial_t u}{|\nabla u|} = \nabla \cdot \left(\frac{\nabla u}{|\nabla u|} \right).
\end{equation*}
That is, the normal velocity of a level-set $\{y \in \R^d: u(y) = c\}$ given by $\partial_t u / |\nabla u|$ is equal to the divergence of the normal $\nabla u /|\nabla u|$ or in other words the mean curvature.  We will denote the associated level-set Laplacian by
\begin{equation*}
    F(\nabla \phi, \nabla^2 \phi) := \nabla \cdot \left(\frac{\nabla \phi}{|\nabla \phi|}\right) |\nabla \phi|= \Delta \phi - \frac{ \nabla \phi}{|\nabla \phi|} \cdot \nabla^2 \phi \frac{\nabla \phi}{|\nabla \phi|} 
\end{equation*}
for any $\phi \in C^2(\R^d)$ with $\nabla \phi \neq 0$. The lower semi-continuous envelope of $F$ reads
\begin{equation*}
  F_*(\nabla \phi, \nabla^2 \phi) = \begin{cases}
    \Delta \phi - \frac{ \nabla \phi}{|\nabla \phi|} \cdot \nabla^2 \phi \frac{\nabla \phi}{|\nabla \phi|} &\text{ if } \nabla \phi \neq 0,\\
    \Delta \phi - \lambda_d &\text{ if } \nabla \phi = 0,
  \end{cases}
\end{equation*}
where $\lambda_d$ is the largest eigenvalue of $\nabla^2 \phi$. The upper semi-continuous envelope is analogously given when replacing $\lambda_d$ with the smallest eigenvalue $\lambda_1$.

Let us recall the definition of a viscosity solution to level-set mean curvature flow:
\begin{definition}\label{def:viscosity}
    A \emph{viscosity super-solution} is a function $u \in C^0(\R^d \times [0, T])$ such that for any $\phi \in C^\infty(\R^d \times [0, T])$ and any minimum point $(x_0, t_0) \in \R^d \times (0,T)$ of $u - \phi$ it holds
    \begin{equation*}
      \partial_t \phi (x_0, t_0) \geq F_*(\nabla \phi(x_0, t_0), \nabla^2\phi(x_0, t_0)).
    \end{equation*}
    Furthermore, $u$ is a \emph{viscosity sub-solution} if $-u$ is a super solution. Finally, for an initial value $u_0 \in C^0(\R^d \times [0, T])$, the function $u$ is a \emph{viscosity solution} to level-set mean curvature flow if $u$ is a viscosity super- and sub-solution and if $u(\cdot, 0) = u_0$.
\end{definition}

In~\cite[Proposition 3.6]{ullrich_median} the limits
\begin{align}
  \limsup_{y \to x, h \downarrow 0} \frac{\med_h \phi - \phi}{h}(y) &\leq F^*(\nabla \phi(x), \nabla^2 \phi(x))\label{eq:consistency_upper},\\
    \liminf_{y \to x, h \downarrow 0} \frac{\med_h \phi - \phi}{h}(y) &\geq F_*(\nabla \phi(x), \nabla^2 \phi(x))\label{eq:consistency_lower}
\end{align}
were already proven when $\phi \in C^{3}(\R^d)$ is an arbitrary test function. In their proof they even showed that for $\nabla \phi \neq 0$ it holds
\begin{equation}\label{eq:strong_consistency}
  \left|\frac{\med_h \phi - \phi}{h}(x) - F(\nabla \phi(x), \nabla^2 \phi(x)) \right| \lesssim \frac{1+ \|\nabla^3 \phi \|_{\infty}}{|\nabla \phi(x)|}\sqrt{h}.
\end{equation}
Their result is for indicator functions of annuli and balls as kernels, but it is easy to adapt to our heat kernel case. Furthermore, for any $\phi, \psi \in C^{0,1}(\R^d \times [0, \infty))$ it holds by Lemma~\ref{lem:uniqueness_median} and a simple computation
\begin{align}
  \phi \leq \psi \Rightarrow \med_h (\phi)(x) \leq \med_h (\psi)(x) &\quad \text{ for any } x \in \Omega, \label{eq:monotonicity_median}\\
  \med_h(\psi + c)(x) = \med_h(\psi)(x) + c &\quad \text{ for any } c \in \R. \label{eq:med_const}
\end{align}
We will combine the proof ideas of~\cite[Theorem 2.1]{zbMATH04205918}, \cite[Theorem 3.3]{zbMATH01336919} and~\cite[Proposition 2.2]{BarlesGeorgelin} to show that the scheme~\eqref{def:semi_discrete_scheme} converges to the unique viscosity solution to level-set mean curvature flow.

\begin{proof}[Proof of Proposition~\ref{the:viscosity}]
  First we record that the comparison principle of~\cite[Theorem 3.2]{evans_spruck_I} yields the uniqueness of viscosity solutions to level-set mean curvature flow.

  Define the two limits 
  \begin{align*}
    \overline{u}(x,t)= \limsup_{y \rightarrow x, s \rightarrow t, h \downarrow 0} u_h(y,s),\\
    \underline{u}(x,t) = \liminf_{y \rightarrow x, s \rightarrow t, h \downarrow 0} u_h(y,s).
  \end{align*}
  We claim that $\overline{u}$ and $\underline{u}$ are viscosity sub- and super-solution, respectively, with $\overline{u}(\cdot, 0) = \underline{u}(\cdot, 0) = u_0$. If the claim is valid, the comparison principle for the viscosity solutions~\cite[Theorem 3.2]{evans_spruck_I} yields $\overline{u} \leq \underline{u}$. But, the opposite inequality is also true by definition of $\underline{u}$ and $\overline{u}$ implying
  \begin{equation*}
    u = \underline{u} = \overline{u}.
  \end{equation*}
  This means $u$ is the unique viscosity solution and by standard arguments~\cite[Remark 6.4]{zbMATH00089226} one concludes the local uniform convergence $u_h \rightarrow u$ in space and time. As $u_h = u = c$ outside $B_L(0)$ this local convergence is upgraded to global uniform convergence on $\R^d \times [0,T]$.

  So it remains to justify the claim.
  
  \emph{Step 1 -- Initial data $\overline{u}(\cdot, 0) = \underline{u}(\cdot, 0) = u_0$:} We focus on the proof of $\overline{u}(\cdot, 0) = u_0$ as the other case is treated similarly. As (see for example~\eqref{eq:Lip_non_increasing} later)
  \begin{equation}\label{eq:uniform_lipschitz_bound}
    \sup_{h \in (0,1], t \in [0, \infty)}\Lip( u_h(\cdot, t)) \leq \Lip(u_0) \leq C_1
  \end{equation}
  holds, we immediately get $\overline{u}(x, 0) = \limsup_{h \downarrow 0, t \downarrow 0} u_h(x,t)$. Furthermore, we have the uniform bound (see for example~\eqref{eq:max_non_increasing} later)
  \begin{equation*}
    \sup_{h \in (0,1], t \in [0, \infty)} \|u_h(\cdot,t)\|_\infty \leq \|u_0(\cdot)\|_\infty \leq C_2.
  \end{equation*}
  Thus, for any sequence $(h_k, t_k) \to (0,0)$ as $k \to \infty$ we can apply the Arzel{\`a}-Ascoli theorem to get a subsequence $(h_{k_n}, t_{k_n}) \to (0,0)$ as $n \to \infty$ such that
  \begin{equation}\label{eq:uniform_v_limit}
    \lim_{n \to \infty} u_{h_{k_n}}(\cdot, t_{k_n}) = v \quad \quad \text{ uniformly},
  \end{equation}
  for some function $v \in C_b^{0,1}(\R^d)$. We claim it has to hold $v = u_0$. By assumption, we have that $u_0 \in C^3_b(\R^d \setminus S)$. As $\R^d \setminus S$ is dense in $\R^d$, the continuity of $v$ and $u_0$ implies that it is enough to prove that $v = u_0$ on $\R^d \setminus S$. So let $x \in \R^d \setminus S$ be given. As $\R^d \setminus S$ is open by assumption, we have that $u_0$ is twice differentiable in a neighborhood of $x$. Furthermore, $u_0$ is bounded. Hence, there exists $\xi_0, \zeta_0 \in C^3_b(\R^d)$ with
  \begin{align*}
    \xi_0 \leq u_0 \leq \zeta_0 \quad \text{ on } \R^d \quad\quad \quad \text{ and } \quad \quad \quad\xi_0(x) = u_0(x) = \zeta_0(x).
  \end{align*} 
  Denote by $\xi_h$ and $\zeta_h$ the solutions to~\eqref{def:semi_discrete_scheme} with initial values $\xi_0$ and $\zeta_0$, respectively.
  By the comparison principle of Lemma~\ref{lem:comp_princ} we get
  \begin{equation*}
    \xi_h \leq u_h \leq \zeta_h.
  \end{equation*}
  Thus, if we can show
  \begin{equation}\label{eq:upper_lower_limits_match}
    \lim_{n \to \infty} \xi_{h_{k_n}}(x, t_{k_n}) = \lim_{n \to \infty} \zeta_{h_{k_n}}(x, t_{k_n}) = u_0(x)
  \end{equation}
  it follows $v(x) = u_0(x)$. We focus on $\lim_{n \to \infty} \xi_{h_{k_n}}(x, t_{k_n}) =$ the version for $\zeta$ is analogous.

  We want to show that the time derivative of $\xi_h$ is uniformly bounded. By Lemma~\ref{lem:mon_time_derivative} it is enough to check this for $t = 0$. Fix $y \in \R^d$ arbitrary and distinguish the following two cases:

  \emph{Case 1 -- Upper Bounded Gradient $|\nabla \xi_0(y)| \leq \sqrt{h}$:}
  Let $R > 0$ be the constant radius with
  \begin{equation*}
    \int_{\{z > R\}} G_h(0,z) \,dz = \frac 12.
  \end{equation*}
  Then we can assume that there exists a $|z_0| \leq R$ with $\med_h \xi_0 (x) = \xi_0(x + z_0)$ as otherwise the Euler-Lagrange inequality~\eqref{eq:ELI1} would be violated. One checks that there is a constant $C$ such that $R \leq C\sqrt{h}$ which yields $|z_0| \leq C \sqrt{h}$.
   By the case distinction and as $\xi_0 \in C^3_b(\R^d)$ we have  
  \begin{equation*}
    |\med_h \xi_0(y) - \xi_0(y)| =  |\xi_0(y+ z_0) - \xi_0(y)| \leq |\nabla \xi_0(y)| |z_0| + \mathcal{O}(|z_0|^2) = \mathcal{O}(h).
  \end{equation*}
  Thus, we have that 
  \begin{equation*}
    |\partial_t \xi_h(y, 0)| = \left| \frac{\med_h \xi_0(y) - \xi_0(y)}{h}\right| \lesssim 1.
  \end{equation*}

  \emph{Case 2 -- Lower Bounded Gradient $|\nabla \xi_0(y)| > \sqrt{h}$:} In this case we apply~\eqref{eq:strong_consistency} to immediately obtain
  \begin{align*}
    |\partial_t \xi_h(y, 0)| \lesssim \frac{1 + \|\nabla^3 \xi_0\|_\infty}{|\nabla \xi_0(y)|} \sqrt{h} + |F(\nabla \xi_0(y), \nabla^2 \xi_0(y))| \leq 1 + \|\nabla^3 \xi_0\|_\infty + \|\nabla^2 \xi_0\|_\infty < \infty.
  \end{align*}

  Combining Case 1, Case 2 and Lemma~\ref{lem:mon_time_derivative} we find
  \begin{equation*}
    \sup_{h \in (0,1], s \in [0, \infty)}\|\partial_t \xi_h(\cdot, s)\|_\infty \lesssim 1.
  \end{equation*}
  But now, as $\xi_{h_{k_n}}(x,0) = u_0(x)$, we simply have
  \begin{equation*}
    |u_0(x) - \xi_{h_{k_n}}(x, t_{k_n})| = \left| \int_{0}^{t_{k_n}} \partial_t \xi_{h_{k_n}}(x, s)\, ds\right| \leq t_{k_n} \sup_{h \in (0,1], s \in [0, \infty)} |\partial_t \xi_h(\cdot, s)| \lesssim t_{k_n}.
  \end{equation*}
  Taking the limit $n \to \infty$ yields~\eqref{eq:upper_lower_limits_match}.

  Now we know that $v = u_0$ and as implication by~\eqref{eq:uniform_v_limit} that
  \begin{equation}
    \lim_{n \to \infty} u_{h_{k_n}}(\cdot, t_{k_n}) = u_0 \quad \quad \text{ uniformly}.
  \end{equation}
  As we have identified the limit, we can get rid of the subsequences which yields
  \begin{equation*}
    \lim_{k \to \infty} u_{h_{k}}(\cdot, t_{k}) = u_0 \quad \quad \text{ uniformly}.
  \end{equation*}
  The sequence $(h_k, t_k) \to 0$ was arbitrary which gives us
  \begin{equation*}
    \limsup_{h \to 0, t \to 0} u_h(\cdot, t) = u_0 \quad \quad \text{ uniformly}.
  \end{equation*}
  Finally, by the Lipschitz continuity of $u_0$ it holds
  \begin{equation*}
    |u_h(y,t) - u_0(x)| \leq |u_h(y,t) - u_0(y)| + \Lip(u_0)|x-y|,
  \end{equation*}
  or in other words 
  \begin{equation*}
    \overline{u}(x,t) = \limsup_{h \to 0, y \to x, t \to 0} u_h(y,t) = u_0(x).
  \end{equation*}

  \emph{Step 2 -- Argument for $\overline{u}$ and $\underline{u}$ are viscosity sub- and super-solutions, respectively:}
  We focus on the proof for $\overline{u}$; the other case is identical. Let $\phi$ be a smooth test function and $(x_0, t_0)$ be a strict local maximum of $\overline{u} - \phi$. W.l.o.g.\ we can assume that $\overline{u}(x_0, t_0) = \phi(x_0, t_0)$, i.e., 
  \begin{equation*}
    \overline{u}(x,t) - \phi(x,t) \leq \overline{u}(x_0, t_0) - \phi(x_0, t_0) = 0 \quad \quad \text{ in } B_r(x_0, t_0),
  \end{equation*}
  and that $\phi \geq 2 \sup_{h} \|u_h\|_{C^0(\R^d \times [0,\infty))}$ outside the ball $B_r(x_0, t_0)$.
  Define the auxiliary functions 
  \begin{align*}
    \psi_\delta (x,y,t) = \overline{u}(x,t) - \frac{|x-y|^4}{\delta} - \phi(y,t),\\
    \psi_{h,\delta} (x,y,t) = u_h(x,t) - \frac{|x-y|^4}{\delta} - \phi(y,t).
  \end{align*}
  We claim: 
  \begin{quote}
    \emph{There exists a sequence $(x_{h,\delta}, y_{h,\delta}, t_{h,\delta})$ of local maximum points of $\psi_{h,\delta}$ that converge to $(x_\delta, x_\delta, t_\delta)$, a maximum point of $\psi_\delta$, as $h \to 0$. We furthermore claim that $(x_\delta, x_\delta, t_\delta)$ converges to $(x_0, x_0, t_0)$ as $\delta \to 0$.}
  \end{quote}
  To see this, we first note that due to $\phi \geq 2 \sup_{h} \|u_h\|_{C^0(\R^d \times [0,T])}$ outside the ball $B_r(x_0, t_0)$ we have for any $y, t \in \R^d \times [0,\infty)$ and $x \in \R^d$ that
  \begin{align*}
    u_h(x_0, t_0) - \frac{|x_0-x_0|^4}{\delta}-\phi(x_0, t_0) = 0 &\geq u_h(x, t) - 2\sup_h \|u_h\|_{C^0(\R^d \times [0,\infty))}\\
     &\geq  u_h(x, t) - \frac{|x-y|^4}{\delta} - \phi(y, t)
  \end{align*} 
  Thus, we can restrict our domain to $(x,y,t) \in \R^d \times B_r(x_0, t_0)$. Additionally, we have for any $(x,y,t) \in B_{r+1}(x_0) \times B_r(x_0, t_0)$ and any $\delta \leq \frac{1}{\|u_h\|_{\infty} + \|\phi\|_{B_r(x_0, t_0)}}$ that
  \begin{equation*}
    u_h(x,t) - \frac{|x-y|^4}{\delta} - \phi(y,t) \leq \|u_h\|_{\infty} - \frac{1}{\delta} + \|\phi\|_{B_r(x_0, t_0)} \leq 0 = u_h(x_0, t_0) - \frac{|x_0-x_0|^4}{\delta}-\phi(x_0, t_0).
  \end{equation*} 
  Hence, there have to exist global minima $(x_{h, \delta}, y_{h, \delta},t_{h, \delta})$ of $\psi_{h, \delta}$ with 
  \begin{equation*}
    (x_{h, \delta}, y_{h, \delta} t_{h, \delta}) \in B_{r+1}(x_0, x_0, t_0).
  \end{equation*}
  This implies (after passing to a subsequence) that $(x_{h, \delta}, y_{h, \delta}, t_{h, \delta}) \to (x_\delta, y_\delta, t_\delta) \in  B_{r+ 1}(x_0, x_0, t_0)$ as $h \to 0$. Taking a further subsequence we can assume that $(x_\delta, y_\delta, t_\delta) \to (\overline{x}, \overline{x}, \overline{t})$ as $\delta \to 0$.
  Let $(\tilde{x}, \tilde{y}, \tilde{t}) \in \R^d \times \R^d \times [0, \infty)$ be arbitrary and choose $(\tilde{x}_h, \tilde{t}_h) \to (\tilde{x}, \tilde{t})$ such that
  \begin{equation}\label{eq:recovery_overline_u}
    \lim_{h \to 0} u_h(\tilde{x}_h, \tilde{t}_h) = \overline{u}(\tilde{x}, \tilde{t}).
  \end{equation}
  By the maximality property of $(x_{h, \delta}, y_{h, \delta}, t_{h, \delta})$ we have
  \begin{equation*}
    u_h(\tilde{x}_h, \tilde{t}_h) -  \frac{|\tilde{x}_h-\tilde{y}|^4}{\delta}-\phi(\tilde{y},\tilde{t}_h)  \leq u_h(x_{h, \delta}, t_{h, \delta}) -  \frac{|x_{h, \delta}-y_{h, \delta}|^4}{\delta}-\phi(y_{h, \delta}, t_{h, \delta}).
  \end{equation*}
  Passing to the limit, using~\eqref{eq:recovery_overline_u} on the left and the definition of $\overline{u}$ on the right yields
  \begin{equation*}
    \overline{u}(\tilde{x}, \tilde{t}) -  \frac{|\tilde{x}-\tilde{y}|^4}{\delta}-\phi(\tilde{y}, \tilde{t}) \leq\overline{u}(x_\delta, t_\delta) -  \frac{|x_\delta-y_\delta|^4}{\delta}-\phi(y_\delta, t_\delta)   ,
  \end{equation*}
  which shows that $(x_\delta, y_\delta, t_\delta)$ is indeed a maximum point of $\psi_\delta$. 
  
  If $(\overline{x}, \overline{t}) \neq (x_0, t_0)$ then the upper semi-continuity of $\overline u$ and the strict maximality of $(x_0, t_0)$ yield the contradiction
  \begin{align*}
    \overline{u}(x_0, t_0) - \phi(x_0, t_0) \leq \limsup_{\delta \to 0}\overline{u}(x_\delta, t_\delta)  - \frac{|x_\delta-y_\delta|^4}{\delta} - \phi(y_\delta,t_\delta) &\leq \overline{u}(\overline{x}, \overline{t}) - \phi(\overline{x}, \overline{t}) \\
    &<  \overline{u}(x_0, t_0) - \phi(x_0, t_0).
  \end{align*}
  This proves our claim. 

  \medskip

  Next, we claim that it is enough to prove
  \begin{equation}\label{eq:sub_solutio_delta}
    \partial_t \phi(y_\delta, t_\delta) \leq F^*(\nabla \phi(y_\delta, t_\delta), \nabla^2 \phi(y_\delta, t_\delta)).
  \end{equation}
  Indeed, after establishing~\eqref{eq:sub_solutio_delta} one can simply pass to the limit $\delta \to 0$ and use the upper semi-continuity of $F^*$ together with the continuity of $\partial_t \phi, \nabla \phi$ and $\nabla^2  \phi$ to conclude
  \begin{equation*}
    \partial_t \phi(y_0, t_0) = \lim_{\delta \to 0} \partial_t \phi(y_\delta, t_\delta) \leq \limsup_{\delta \to 0} F^*(\nabla \phi(y_\delta, t_\delta), \nabla^2 \phi(y_\delta, t_\delta)) \leq  F^*(\nabla \phi(x_0, t_0), \nabla^2 \phi(x_0, t_0)).
  \end{equation*}

  It remains to prove~\eqref{eq:sub_solutio_delta} for which we distinguish to cases:

  \emph{Case 1 -- Non Singular $\nabla_y \phi (y_{\delta}, t_{\delta}) \neq 0$:} In this case we can assume that $|\nabla \phi (y, t)| > a > 0$ for $y, t$ close to $y_\delta, t_\delta$. By~\eqref{eq:strong_consistency}, we know that the convergence speed of the consistency~\eqref{eq:consistency_upper} only depends on $\frac{1}{|\nabla \phi|}$ which is here uniformly bounded. Hence, the space consistency~\eqref{eq:consistency_upper} is even a space time consistency
  \begin{equation}\label{eq:space_time_consistency}
    \liminf_{y \to y_\delta, h \to 0, t \to t_\delta} \frac{\med_h \phi(\cdot, t) - \phi(\cdot, t)}{h}(\tilde{y}) \leq F^*(\nabla \phi(y_\delta), \nabla^2 \phi(y_\delta)).
  \end{equation}

  Note that $(y_{h, \delta}, t_{h, \delta})$ is a maximum point of
  \begin{equation*}
    (y,t) \mapsto u(y - (y_{h, \delta} - x_{h, \delta}),t) - \frac{|x_{h, \delta}- y_{h, \delta}|^4}{\delta} - \phi(y, t).
  \end{equation*}
  This implies
  \begin{align*}
    \partial_t u_h(x_{h,\delta}, t_{h,\delta}) &=  \partial_t \phi(y_{h,\delta}, t_{h,\delta}),\\
    u(y - (y_{h, \delta} - x_{h, \delta}),t_{h,\delta}) &\leq \phi(y, t_{h,\delta}) + \xi_\delta \quad \quad \text{ for all } y \in \R^d,\\
    u(x_{h, \delta},t_{h,\delta}) &= \phi(y_{h, \delta}, t_{h,\delta}) + \xi_\delta,
  \end{align*}
  where we denoted $\xi_\delta :=  u(x_{h, \delta},t_{h,\delta}) - \phi(y_{h, \delta}, t_{h,\delta})$. By definition of our scheme~\eqref{def:semi_discrete_scheme}, the monotonicity~\eqref{eq:monotonicity_median}, and the invariance of constant translation~\eqref{eq:med_const} it holds
  \begin{align*}
    0 &= \partial_t u_h(x_{h, \delta}, t_{h, \delta}) - \frac{\med_h u_h(\cdot - (y_{h, \delta} - x_{h, \delta}), t_{h, \delta}) - u_h(\cdot - (y_{h, \delta} - x_{h, \delta}), t_{h, \delta})}{h}(y_{h, \delta})\\
    & \geq \partial_t \phi(y_{h, \delta}, t_{h, \delta}) - \frac{\med_h(\phi(\cdot, t_{h, \delta}) -\xi_\delta) - (\phi(\cdot, t_{h, \delta}) + \xi_\delta)}{h}(y_{h, \delta})\\
    &= \partial_t \phi(y_{h, \delta}, t_{h, \delta}) - \frac{\med_h \phi(\cdot, t_{h, \delta})  - \phi(\cdot, t_{h, \delta}) }{h}(y_{h, \delta}).
  \end{align*}
  Taking the limes inferior over $h$ and utilizing~\eqref{eq:space_time_consistency} yields the desired~\eqref{eq:sub_solutio_delta}:
  \begin{align*}
    0 &\geq \liminf_{h \downarrow 0} \partial_t \phi(y_{h, \delta}, t_{h, \delta}) - \frac{\med_h \phi(\cdot, t_{h, \delta})  - \phi(\cdot, t_{h, \delta}) }{h}(y_{h, \delta})\\
    &\geq \liminf_{y \to y_\delta, h \to 0, t \to t_\delta} \partial_t \phi(y, t) - \frac{\med_h \phi(\cdot, t)  - \phi(\cdot, t) }{h}(y)\\
    &\geq \partial_t \phi(y_\delta, t_\delta) - F^*(\nabla \phi(y_\delta), \nabla^2 \phi(y_\delta)).
  \end{align*}
  
  \emph{Case 2 -- Singularity $\nabla_y \phi (y_{\delta}, t_{\delta}) = 0$:}
  As $(x_{ \delta}, y_{ \delta}, t_{ \delta})$ is a local maximum point of $\psi_{ \delta}$ we have $0 = \nabla_y \psi_{ \delta}(x_{ \delta}, y_{ \delta}, t_{ \delta})$ and $0 \geq  \nabla^2_y \psi_{ \delta}(x_{ \delta}, y_{ \delta}, t_{ \delta})$ which imply
  \begin{equation}\label{eq:first_order_after_limit}
    \nabla_y \phi(y_{ \delta}, t_{ \delta}) = \frac{4 (x_{ \delta}- y_{\delta})|x_{ \delta}- y_{\delta}|^2}{\delta}
  \end{equation}
  and 
  \begin{equation}\label{eq:second_order_after_limit}
    \nabla^2_y \phi(y_{ \delta}, t_{\delta}) \geq - \frac{4 |x_{ \delta}- y_{\delta}|^2}{\delta}\mathop{Id} - \frac{8 (x_{ \delta}- y_{h,\delta})\otimes (x_{ \delta}- y_{\delta})}{\delta}.
  \end{equation}
  Hence, our case $\nabla_y \phi (y_{ \delta}, t_{ \delta}) = 0$ yields first $x_{\delta} = y_{ \delta}$ and then
  \begin{equation*}
    \nabla_y^2 \phi(y_{ \delta}, t_{\delta}) \geq 0.
  \end{equation*} 
  Define the function $\tilde{\phi}(x, t) := \frac{|x- y_{h, \delta}|^4}{\delta} + \phi(y_{h, \delta}, t)$. Recall that $(x_{h, \delta}, t_{h, \delta})$ is a local maximum point of $\psi_{h, \delta}(\cdot, y_{h,\delta}, \cdot ) = u_h -\tilde{\phi}$. Thus, we get
  \begin{align*}
    \partial_t u_h(y_{h, \delta}, t_{h,\delta}) &= \phi(y_{h, \delta}, t_{h,\delta}),\\
    u_h(x,t) &\leq \frac{|x- y_{h, \delta}|^4}{\delta} + \phi(y_{h, \delta}, t) + \xi_{\delta} \quad \quad \text{ for all } (x,t) \in \R^d \times [0,\infty),\\
    u_h(x_{h, \delta}, t_{h, \delta}) &=  \frac{|x_{h, \delta}- y_{h, \delta}|^4}{\delta} + \phi(y_{h, \delta}, t_{h,\delta})  +\xi_{\delta},
  \end{align*}
  where we denoted $\xi_{\delta} := u_h(x_{h, \delta}, t_{h, \delta}) - \frac{|x_{h, \delta}- y_{h, \delta}|^4}{\delta} -\phi(y_{h, \delta}, t_{h, \delta})$. The trick is now that the space variable $x$ and time variable $t$ are split in $\frac{|x- y_{h, \delta}|^4}{\delta} + \phi(y_{h, \delta}, t) $, so we can pull the $\phi(y_{h, \delta}, t) $ term out of the median. To be more precise, similar to the first case we compute
  \begin{align*}
    0 &= \partial_t u_h(x_{h, \delta}, t_{h, \delta}) - \frac{\med_h u_h(\cdot, t_{h, \delta}) - u_h(\cdot, t_{h, \delta})}{h}(x_{h, \delta})\\
    & \geq \partial_t \phi(x_{h, \delta}, t_{h, \delta}) - \frac{\med_h(\frac{|\cdot- y_{h, \delta}|^4}{\delta} + \phi(y_{h, \delta},  t_{h, \delta}) + \xi_{\delta}) - (\frac{|\cdot- y_{h, \delta}|^4}{\delta} + \phi(y_{h, \delta},  t_{h, \delta}) + \xi_{\delta})}{h}(x_{h, \delta})\\
    &= \partial_t \phi(x_{h, \delta}, t_{h, \delta}) - \frac{\med_h \frac{|\cdot- y_{h, \delta}|^4}{\delta}  - \frac{|\cdot- y_{h, \delta}|^4}{\delta} }{h}(x_{h, \delta}).
  \end{align*}
  Now we simply take the limes inferior in $h$ again and use the only in space dependent consistency~\eqref{eq:consistency_upper} together with 
  \begin{equation}\label{eq:vanish_diff_x_y}
    \lim_{h \to 0} x_{h, \delta} - y_{h, \delta} = x_\delta - y_\delta = 0, 
  \end{equation}
  and 
  \begin{equation}\label{eq:F_mon_happy}
    F^*(0,0) \stackrel{\eqref{eq:first_order_after_limit}}{=} F^*(\nabla_y \phi(y_\delta, t_\delta), 0) \stackrel{\eqref{eq:second_order_after_limit}}{\leq} F^*(\nabla_y \phi(y_\delta, t_\delta), \nabla^2_y \phi(y_\delta, t_\delta))
  \end{equation}
  to conclude
  \begin{align*}
    0 &\geq \liminf_{h \to 0} \partial_t \phi(x_{h, \delta}, t_{h, \delta}) - \frac{\med_h \frac{|\cdot|^4}{\delta}  - \frac{|\cdot|^4}{\delta} }{h}(x_{h, \delta} - y_{h, \delta})\\
    &\stackrel{\eqref{eq:vanish_diff_x_y}}{\geq} \liminf_{z \to 0, h \to 0}\partial_t \phi(x_{h, \delta}, t_{h, \delta}) - \frac{\med_h \frac{|\cdot|^4}{\delta}  - \frac{|\cdot|^4}{\delta} }{h}(z)\\
    &\stackrel{\eqref{eq:consistency_upper}}{\geq} \partial_t \phi(x_{ \delta}, t_{ \delta}) - F^*(0, 0)\\
    &\stackrel{\eqref{eq:F_mon_happy}}{\geq} \partial_t \phi(y_{ \delta}, t_{ \delta}) - F^*(\nabla_y \phi(y_\delta, t_\delta), \nabla^2_y \phi(y_\delta, t_\delta)).
  \end{align*}
  This finishes the second case and with it the whole proof. 
\end{proof}

\section{Energy Convergence -- Proof of Main Result}\label{sec:conv_energie}
In this section, we prove the main result, namely the convergence of the energies. In other words, we want to make the heuristic argument of Section~\ref{sec:main_result} rigorous. To that aim we need to prove a uniform bound on the $L^1$ norm of $H_h$. The following two approximations will be necessary:
\begin{itemize}
  \item[1.] To be able to work with local convergences it will be convenient to cut off the (heat-) kernel so that it has compact support. We will denote this, at the radius $R >0$ cut-off kernel by 
  \begin{equation*}
    G_{h, R}(x,y) := \mathbbold{1}_{|x-y| \leq R} G_h(x,y).
  \end{equation*}
  Then, we define the approximate scheme by
     \begin{equation}\label{def:uhR}
    \begin{cases}
      \partial_t u_{h, R} = \frac{\med_{h,R} u_{h, R} - u_{h, R}}{h} & \text{in } \R^d \times (0,\infty),\\
      u_{h, R}(\cdot, t) = u_0 &\text{on } \R^d \times \{t = 0\}
    \end{cases}
   \end{equation}
   where the cut-off median is
   \begin{equation*}
    \med_{h,R} f (x) \in \argmin_{c \in \R} \int_{\R^d} G_{h,R}(x,y) |c - f(y)| \, dy.
   \end{equation*} 
   \item[2.] As in Evans and Spruck's work~\cite{evans_spruck_IV} it is necessary to regularize the $\sign$ function in order to make $H_h$ differentiable. Consequently, we introduce the function $\Psi_\varepsilon \in C^2$ such that 
  \begin{align}\label{def:psi_eps}
    \psi_\eps := \Psi_\eps' \to \sign \quad \text{pointwise as } \eps \downarrow 0, \quad \quad \quad \frac{C_1}{\eps} \geq \psi_\varepsilon' \geq 0 \text{ on } \R,\quad   \psi_\varepsilon' > \frac{c_1}{\eps} \text{ on } \bigg[- \frac{\eps}{2}, \frac{\eps}{2}\bigg],\\
     -\psi_\varepsilon(s) = \psi_\varepsilon(-s), \quad \quad |\psi_\eps| \equiv 1 \text{ on } \R\setminus [-\eps, \eps], \nonumber
  \end{align}
  where $c_1> 0$ and $C_1>0$ are constants independent of $\eps$. 
  
   Define $u_{h, \eps, R}$ by satisfying
   \begin{equation}
    \begin{cases}\label{def:u_heps}
      \partial_t u_{h, \eps, R} = \frac{\med_{h, \eps, R} u_{h, \eps, R} - u_{h, \varepsilon, R}}{h} & \text{in } \R^d \times (0,\infty),\\
      u_{h, \eps, R}(\cdot, t) = u_0 &\text{on } \R^d \times \{t = 0\}
    \end{cases}
   \end{equation}
   where the median for any function $f$ with respect to $\Psi_\eps$ (and additionally cut-off kernel $G_{h,R}$) is naturally given by
   \begin{equation}\label{def:epsmed}
    \med_{h, \eps, R} f (x) \in \argmin_{c \in \R} \int_{\R^d} G_{h,R}(x,y) \Psi_\eps(c - f(y)) \, dy.
   \end{equation} 
   Furthermore, we define the associated approximate mean curvature by 
   \begin{equation*}
    H_{h,\eps, R}(x) := \frac{2}{\sqrt{h}} \int_{\R^d} G_{h,R}(x,y) \psi'_{\eps}(u_{h,\eps,R}(x) - u_{h, \eps,R}(y)) \, dy.
   \end{equation*}
\end{itemize}

\begin{lemma}\label{lem:mono_energie}
    The energy $E_h(u_h)$ is non-increasing in time, i.e., for any $0 \leq t_1 \leq t_2 < \infty$ it holds
    \begin{equation*}
        E_h(u_h(\cdot, t_1)) \geq E_h(u_h(\cdot, t_2)).
    \end{equation*}
\end{lemma}
\begin{proof}
    \emph{Step 1 -- The $\eps$-regularized case}: We claim the lemma is true when we use the energy and median scheme with the $\eps$ regularized modulus, i.e.,
    \begin{equation}\label{eq:energy_mono_eps}
        \frac{d}{dt} \frac{1}{\sqrt{h}} \int_{\R^d}\int_{\R^d} G_{h,R}(x,y) \Psi_\eps(u_{h,\eps, R}(x) - u_{h, \eps, R}(y)) \, dy \, dx   \leq 0.
    \end{equation}

    Indeed, computing the time derivative yields using non-local integration by parts
    \begin{multline*}
         \frac{d}{dt} \frac{1}{\sqrt{h}} \int_{\R^d}\int_{\R^d} G_{h,R}(x,y) \Psi_\eps(u_{h,\eps, R}(x) - u_{h, \eps, R}(y)) \, dy \, dx \\
         =  \int_{\R^d} H_{h, \eps, R}(x)  \frac{\med_{h,\eps, R} u_{h,\eps,R}(x) - u_{h, \eps, R}(x)}{h}\, dx.
    \end{multline*}
    We now claim that for any $x \in \R^d$ we have the pointwise estimate
    \begin{equation}\label{eq:different_signs_eps_h_r}
        H_{h, \eps, R}(x)  \big(\med_{h,\eps, R} u_{h,\eps,R}(x) - u_{h, \eps, R}(x)\big) \leq 0
    \end{equation}
    
To see that~\eqref{eq:different_signs_eps_h_r} is indeed true, fix $x \in \R^d$. Assume (the other case is similar) that 
\begin{equation*}
  \med_{h,\eps,R} u_{h,\eps,R}(x) - u_{h,\eps,R}(x) \geq 0.
\end{equation*}
Then, it has to hold $H_{h,\eps,R}(x) \leq 0$ as otherwise one would have the contradiction (using the monotonicity of $\psi_\eps$with $\med_{h, \eps, R} u_{h,\eps,R}(x) > u_{h,\eps,R}(x)$)
\begin{align*}
  0 < H_{h,\eps,R}(x)& = \frac{2}{\sqrt{h}} \int_{\R^d} G_{h, R}(x,y) \psi_\eps(u_{h,\eps,R}(x) - u_{h,\eps,R}(y)) \,dy\\
  &\leq \frac{2}{\sqrt{h}} \int_{\R^d} G_{h,R}(x,y)  \psi_\eps(\med_{h, \eps, R} u_{h,\eps,R}(x) - u_{h,\eps,R}(y)) \,dy =0.
\end{align*} 
Here we used in the last equation the Euler-Lagrange equation, resulting from the differentiability of $\Psi_\eps$ in~\eqref{eq:ELI1} and~\eqref{eq:ELI2}.
Thus,~\eqref{eq:different_signs_eps_h_r} has to hold, which in turn implies~\eqref{eq:energy_mono_eps}.

\emph{Step 2 -- Taking the limit $\eps \downarrow 0$}: Integrating~\eqref{eq:energy_mono_eps} in time yields for any $0 \leq t_1 \leq t_2 < \infty$
\begin{align}
     \frac{1}{\sqrt{h}} \int_{\R^d}\int_{\R^d} G_{h,R}(x,y) \Psi_\eps(u_{h,\eps, R}(x, t_1) - u_{h, \eps, R}(y, t_1)) \, dy \, dx \label{eq:integrated_mono}\\
     \geq  \frac{1}{\sqrt{h}} \int_{\R^d}\int_{\R^d} G_{h,R}(x,y) \Psi_\eps(u_{h,\eps, R}(x, t_2) - u_{h, \eps, R}(y, t_2)) \, dy \, dx. \nonumber
\end{align}
Since $\lim_{\eps \downarrow 0} u_{h,\eps,R}  = u_{h,R}$ uniformly on $\R^d \times \{t = t_i\}$, $i = 1,2$, by Lemma~\ref{lem:uniform_convergenceReps}, $\Psi_\eps \rightarrow |\cdot|$ uniformly, and $u_{h,\eps,R}$ has exponential decay by Lemma~\ref{lem:exp_decay}, we have
\begin{multline*}
    \lim_{\eps \downarrow 0} \int_{\R^d}\int_{\R^d} G_{h,R}(x,y) \Psi_\eps(u_{h,\eps, R}(x, t_i) - u_{h, \eps, R}(y, t_i)) \, dy \, dx \\
    = \int_{\R^d}\int_{\R^d} G_{h,R}(x,y) |u_{h, R}(x, t_i) - u_{h, R}(y, t_i)| \, dy \, dx.
\end{multline*}
Similarly, using $\lim_{R\uparrow \infty} u_{h, R} = u_{h}$ uniformly by Lemma~\ref{lem:uniform_convergenceReps} and $u_{h,R}$ is constant outside a large ball $B_L(0)$ by Lemma~\ref{lem:const_outside_ball} we have
\begin{equation*}
    \lim_{R\uparrow \infty} \int_{\R^d}\int_{\R^d} G_{h,R}(x,y) |u_{h, R}(x, t_i) - u_{h, R}(y, t_i)| \, dy \, dx\\
     = \int_{\R^d}\int_{\R^d} G_{h}(x,y) |u_{h}(x, t_i) - u_{h}(y, t_i)| \, dy \, dx.
\end{equation*}
Hence, in~\eqref{eq:integrated_mono} we can pass first to the limit $\eps \to 0$ and then $R \to \infty$ which concludes the proof of this lemma.
\end{proof}

Now for the approximation of the scalar mean curvature we have the crucial $L^1$ bound.
\begin{proposition}\label{lem:monotonicity_Hh}
  It holds for any $0 \leq t_1 \leq t_2$ that 
  \begin{equation*}
    \int_{\R^d} |H_{h,\eps, R}(x,t_2)| \, dx \leq \int_{\R^d} |H_{h,\eps,R}(x,t_1)| \, dx.
  \end{equation*}
\end{proposition}

A simple compactness argument will then use this $L^1$ bound to get a well-behaved limit of the approximate mean curvature. Note, the following lemma that states this convergence, is written in form of subsequences, because we will later  carefully pick a $\eps(h)$ and $R(h)$ dependent on $h$.
\begin{lemma}\label{lem:limit_of_Hh}
 For any sequence $(h_k, \eps_k, R_k)_{k\in \N}$ converging to $(0,0, \infty)$ and satisfying the bound
 \begin{equation*}
  \sup_{k}  \int_{\R^d} |H_{h_k, \eps_k, R_k}(x, t = 0) |\,dx < \infty,
 \end{equation*} 
  there exists a measure $\mu$ such that
  \begin{equation*}
    H_{h_k, \eps_k, R_k} dx \,dt \xrightharpoonup{*} \mu
  \end{equation*}
  as measures in time and space. Furthermore, there exists a measurable vector field $\sigma:\R^d \times [0, \infty) \to \R^d$ such that $|\sigma| \leq 2c_0 = \frac{2}{\sqrt{\pi }}$ and $\mu = \nabla \cdot \sigma \, dx \,dt$.
\end{lemma}
  
\begin{proof}[Proof of Proposition \ref{lem:monotonicity_Hh}]  
   \emph{Step 1 -- Introducing $\delta$:} Take $\eta_\delta \in C^1(\mathbb{R}, \mathbb{R})$ with $\eta_\delta \to |\cdot|$ uniformly as $\delta \downarrow 0$ and $\eta'(s) = \sign(s)$ for $|s| > \delta$ as well as $|\eta'| \leq 1$. For simpler readability we drop the $h,\eps,R$ subscript for $H_{h, \eps, R}$, $u_{h,\eps,R}$ and $\med_{h,\eps,R}$, i.e., we write $\Hher$ for $H_{h, \eps, R}$, $\uher$ for $u_{h,\eps,R}$ and $\med$ for $\med_{h,\eps,R}$. Then the derivative of the $L^1$-norm with the modulus replaced by $\eta_\delta$ computes as (suppressing the $t$-dependence on the right-hand side for notational convenience)
  \begin{align*}
    &\frac{d}{dt} \int_{\R^d} \eta_\delta (\Hher(x,t)) \, dx \\
    &= \int_{\R^d} \eta_\delta '(\Hher(x)) \int_{\R^d} \frac{1}{\sqrt{h}} G_{h,R}(x,y) \psi_\eps'\big(\uher(x) - \uher(y)\big) \left(\frac{\med \uher(x) - \uher(x)}{h} - \frac{\med \uher(y) - \uher(y)}{h} \right)\!\! \, dy \, dx\\
    &= \frac{1}{h\sqrt{h}} \int_{\R^d} \int_{\R^d} \big( \eta_\delta '(\Hher(x)) - \eta_\delta '(\Hher(y)) \big)   G_{h,R}(x,y) \psi_\eps'\big(\uher(x) - \uher(y)\big) \big(\med \uher(x) - \uher(x)\big) \, dy \, dx.
  \end{align*}

  \emph{Step 2 -- Opposite Signs:} We claim, if $|\Hher(x,t)| \geq \delta$  then it holds
  \begin{equation}\label{eq:sign_of_integrand}
    \big( \eta_\delta '(\Hher(x,t)) - \eta_\delta '(\Hher(y,t)) \big)   G_{h,R}(x,y) \psi_\eps'(\uher(x) - \uher(y)) (\med \uher(x) - \uher(x)) \leq 0,
  \end{equation}
  and if $|\Hher(x,t)| = 0 $ then $\med  \uher (x) = \uher(x)$.
  
  We focus on the case $\Hher(x,t) \geq 0$; the other sign works analogously. Note the Euler-Lagrange inequalities~\eqref{eq:ELI1}, \eqref{eq:ELI2} yield an Euler-Lagrange equality by the differentiability of $\Psi_\eps$. Hence, by an analogous computation to~\eqref{eq:sign_trick} it holds $\Hher(x,t) = 0$ iff $\med \uher(x) = \uher(x)$ and $\Hher(x,t) > 0$ iff $\med \uher(x) < \uher(x)$. Hence, if $\Hher(x,t) \geq \delta$ then it holds $\med \uher(x) < \uher(x)$ and $\eta_\delta'(\Hher(x,t)) = 1$  implies $\eta_\delta '(\Hher(x,t)) - \eta_\delta '(\Hher(y,t)) \geq 0$ due to $|\eta_\delta'| \leq 1$. As $G_{h,R} \geq 0$ and $\psi_\eps' \geq 0$ it follows that the inequality \eqref{eq:sign_of_integrand} is satisfied in the case  $\Hher(x) \geq \delta$.

\emph{Step 3 -- Integrating in Time:} Combining Step 1 and Step 2 yields
\begin{align}
  &\frac{d}{dt} \int \eta_\delta (\Hher(x,t)) \, dx \label{eq:comb_steps}\\
  &\leq \frac{1}{h\sqrt{h}}\int_{ \{|\Hher(x,t)| < \delta\}} \int_{\R^d} \big( \eta_\delta '(\Hher(x,t)) - \eta_\delta '(\Hher(y,t)) \big)   \\
  &\hspace{100pt} \times G_{h, R}(x,y) \psi_\eps'(\uher(x,t) - \uher(y,t)) \left(\med \uher(x,t) - \uher(x,t)\right) \, dy \, dx.\nonumber
\end{align}
%As $\uher$ is constant outside a big ball $B_{L}(0)$ and $ G_{h,R}(x,y) = 0$ for $|x-y|>R$ it holds 
%\begin{equation*}
%  G_{h,R}(x,y)\psi_\eps(\uher(x) - \uher(y)) = 0  \quad \quad \text{for any } (x,y) \in (\R^{2d} \setminus B^2_{R+L}(0)).
%\end{equation*}

Integrating~\eqref{eq:comb_steps} in time yields
\begin{align}\label{eq:comb_integrated}
  &\int \eta_\delta(\Hher(x,t_2)) \, dx \\
  &\leq \int \eta_\delta(\Hher(x,t_1)) \, dx + \frac{1}{h\sqrt{h}}  \int_{t_1}^{t_2} \int_{\{|\Hher(x,t)| < \delta\}} \big( \eta_\delta '(\Hher(x,t)) - \eta_\delta '(\Hher(y,t)) \big) \nonumber\\
  &\hspace{140pt} \times G_{h, R}(x,y) \psi_\eps'(\uher(x,t) - \uher(y,t)) \left(\med \uher(x,t) - \uher(x,t)\right) \, dy \, dx \, dt. \nonumber
\end{align}
As, we have no control of the prefactor $\frac{1}{h\sqrt{h}}$ in the last term of the right-hand side, we have to show that the whole term vanishes as $\delta \downarrow 0$ similarly as in Evans and Spruck~\cite{evans_spruck_IV}. The approximation $\eta_\delta$ was specifically chosen for this purpose.

\emph{Step 4 -- the limit $\delta \downarrow 0$:} Combining the fact that $\Hher(x,t)$ has exponential decay as $|x| \rightarrow \infty$ by Lemma~\ref{lem:exp_decay} with the uniform convergence of $\eta_\delta \to |\cdot|$ we conclude at any time $t$
\begin{equation}\label{eq:delta_energy}
  \lim_{\delta \downarrow 0} \int_{\R^d} \eta_\delta(\Hher(x,t)) \, dx \rightarrow \int_{\R^d}|\Hher(x,t)|\,dx .
\end{equation}
As $\big| \eta_\delta '(\Hher(x,t)) - \eta_\delta '(\Hher(y,t)) \big| \leq 2$ there exists a $\xi \in L^\infty(\R^d \times \R^d \times [t_1, t_2])$ such that (up to subsequences)
\begin{equation}\label{eq:conv_etaprime}
   \eta_\delta '(\Hher(x,t)) - \eta_\delta '(\Hher(y,t)) \xrightharpoonup{*} \xi(x,y,t) \quad \quad \text{ weakly-}* \text{ in } L^\infty(\R^d \times \R^d \times [0, \infty)) \text{ as } \delta \downarrow 0.
\end{equation}
Due to the dominated convergence theorem we have for any compact $K \subset \R^d \times [0,T)$ that
\begin{equation}\label{eq:indicator_delta}
  \lim_{\delta \downarrow 0}\left|\int_{K} \mathbbold{1}_{\{|H(x,t)| < \delta\}} \,dx \,dt - \int_{K} \mathbbold{1}_{\{ |\Hher(x,t)| = 0\}} \,dx \,dt \right| = \int_{K} \lim_{\delta \downarrow 0} \mathbbold{1}_{\{0 < |\Hher(x,t)| < \delta\}} \,dx \,dt = 0.
\end{equation}
In other words $ \mathbbold{1}_{\{ |\Hher(x,t)| < \delta\}}$ converges in $L^1_{loc}(\R^d \times[0, \infty))$. 

Furthermore, due to the exponential decay of $u$ and $\med u$ see~\eqref{eq:exp_u} and~\eqref{eq:exp_medu}, we can control the integrand of the right-hand side of~\eqref{eq:comb_integrated} by
\begin{align*}
   &\left|  \int_{\R^d} \big( \eta_\delta '(\Hher(x,t)) - \eta_\delta '(\Hher(y,t)) \big) G_{h, R}(x,y) \psi_\eps'(\uher(x) - \uher(y)) \left(\med \uher(x) - \uher(x)\right) \, dy \right|\\
   &\lesssim 4 |B_R(0)| G_h(0,0) \|\psi''_\eps\|_\infty e^{-(|x| - L)_+}.
\end{align*}
Combining this exponential decay with the local weak-$*$ convergence~\eqref{eq:conv_etaprime} and local strong convergence~\eqref{eq:indicator_delta}, we can also pass to the limit in the right-hand side term of~\eqref{eq:comb_integrated}
\begin{align}
  &\lim_{\delta \downarrow 0}\int_{t_1}^{t_2} \int_{\{ |\Hher(x,t)| < \delta\} } \int_{\R^d} \big( \eta_\delta '(\Hher(x,t)) - \eta_\delta '(\Hher(y,t)) \big) \label{eq:delta_vanish}\\
  &\hspace{150pt} \times G_{h, R}(x,y) \psi_\eps'(\uher(x) - \uher(y)) \left(\med \uher(x) - \uher(x)\right) \, dy \, dx \,dt\nonumber\\
  &= \int_{t_1}^{t_2} \int_{\{|\Hher(x,t)| = 0\}}\int_{\R^d}\xi(x,y,t) G_{h, R}(x,y) \psi_\eps'(\uher(x) - \uher(y)) \left(\med \uher(x) - \uher(x)\right) \, dy \, dx \, dt \nonumber\\
  &= 0, \nonumber
\end{align}
where we used in the last line the observation of \emph{Step 2} that $\Hher(x,t) = 0$ if and only if $\med \uher(x) = \uher(x)$.
Combining our observations~\eqref{eq:delta_energy}, \eqref{eq:delta_vanish} lets us pass to the limit in~\eqref{eq:comb_integrated} to obtain the claimed
\begin{equation*}
  \int_{\R^d} |\Hher(x,t_2)| \, dx \leq \int_{\R^d} |\Hher(x,t_1)| \, dx. \qedhere
\end{equation*}
\end{proof}
\begin{proof}[Proof of Lemma \ref{lem:limit_of_Hh}]
  Due to Proposition \ref{lem:monotonicity_Hh} it holds 
  \begin{align*}
    \sup_{k} \int_0^T \int_{\R^d} \left|H_{h_k, \eps_k, R_k} \right| \, dx \, dt \leq T \sup_{k} \int_{\R^d} \left|H_{h_k, \eps_k, R_k}(x, t = 0) \right| \, dx
    < \infty
  \end{align*}
  which directly implies (up to subsequences) the existence of the claimed limit $\mu$. So it only remains to prove the existence of a vector field $\sigma$ such that $\mu = \nabla \cdot \sigma \, dx \, dt$. To that aim, take a smooth, compactly supported test function $\xi$ and compute
  \begin{multline*}
    \int_0^T \int \xi H_{h_k, \eps_k, R_k} \,dx \, dt \\
    = \int_0^T \int_{\R^d} \int_{\R^d} \frac{\xi(x,t) - \xi(x- \sqrt{h_k}z,t)}{\sqrt{h_k}} G_{h_k, R_k}(z) \psi_{\eps_k}'\Big(u_{h_k, \eps_k, R_k}(x,t) - u_{h_k, \eps_k, R_k}(x- \sqrt{h_k}z)\Big)\, dx \, dz \, dt.
  \end{multline*} 
  As $|\psi_{\eps_k}' | \leq 1$, there exist a function $\tilde{\sigma}$ with $|\tilde{\sigma}| \leq 1$ such that 
  \begin{equation*}
    \psi_{\eps_k}'\big(u_{h_k, \eps_k, R_k}(x,t) - u_{h_k, \eps_k, R_k}(x- \sqrt{h_k}z) ,t\big) \xrightharpoonup{*} \tilde{\sigma}(x,z,t) \text{ weakly-$*$ in } L^\infty(\R^d \times \R^d \times [0, \infty)).
  \end{equation*}
  By the uniform convergence of the difference quotient of $\xi$, the compact support of $\xi$, and the exponential decay of the heat kernel, this implies 
  \begin{equation*}
    \int_0^T \int_{\R^d} \xi \, d\mu  = \lim_{k \uparrow \infty} \int_0^T \int \xi H_{h_k, \eps_k, R_k} \,dx \, dt = \int_0^T \int_{\R^d} \int_{\R^d} \nabla \xi(x) \cdot z G_1(z) \tilde{\sigma}(x,z,t) \, dz \, dx\, dt.
  \end{equation*}
  Hence, after defining $\sigma(x,t) := \int z G_1(z) \tilde{\sigma}(x,z,t) \,dz$ it holds the claimed differential relation $\mu = \nabla \cdot \sigma \, dx\, dt$. Furthermore, for any $e \in \mathbb{S}^{d-1}$ it holds 
  \begin{align*}
    e \cdot \sigma(x) &= e \cdot \int z G_1(z) \tilde{\sigma}(x,z) \, dz\\
    &= \int_{\{e \cdot z > 0\}}e \cdot z G_1(z)  (\tilde{\sigma}(x, z) - \tilde{\sigma}(x,-z)) \,dz \\
    &\leq 2 \int_{z_1 > 0} G_1(z) z_1 \, dz \\
    &= 2c_0,
  \end{align*}
  which implies the claimed bound $|\sigma| \leq 2c_0$.
\end{proof}

Having all the important tools in form of Proposition~\ref{lem:monotonicity_Hh} and Lemma~\ref{lem:limit_of_Hh} we can puzzle everything together to obtain our main Theorem~\ref{cor:energy_conv}.
\begin{proof}[Proof of Theorem \ref{cor:energy_conv}]
  We start with the upper bound. Let $\delta$ be fixed but arbitrary. %Recall that $L$ is the constant such that $u_h$ is constant outside of $B_L(0)$. 
  %Choose $K> 0$ such that
  %\begin{equation*}
  %  \sup_{h \in (0, 1]} \frac{1}{\sqrt{h}} \int_{z > K} G_h(0,z) \,dz \leq \frac{\delta_1}{\omega_d L^d\|u_0\|_\infty},
  %\end{equation*} 
 % what is possible due to the exponential decay of the heat kernel. 
  Let us introduce $\tilde{\Psi}_\eps(s) := \psi_\eps(s) s$ which converges uniformly to the modulus $|\cdot|$ as $\eps \downarrow 0$ and satisfies the bound $0 \leq \tilde{\Psi}_\eps(s) \leq |s|$. 
  By Lemma~\ref{lem:uniform_convergenceReps} and the uniform convergence of $\tilde{\Psi}_\eps$ it holds, using discrete integration by parts~\eqref{eq:discrete_IBP_example} and Fatou's lemma (applicable as $\tilde{\Psi}_\eps \geq 0$)  that
  \begin{align}
     \liminf_{R \uparrow \infty } \liminf_{\eps \downarrow 0}&\int_{\R^d} H_{h, \eps, R}\, u_{h,\eps, R} \,dx \nonumber\\
     &= \liminf_{R \uparrow \infty }  \liminf_{\eps \downarrow 0} \frac{1}{\sqrt{h}} \int_{\R^d} \int_{\R^d} G_{h, R}(x,y) \tilde{\Psi}_\eps(u_{h, \eps, R}(x) - u_{h, \eps, R}(y)) \, dy \,dx. \nonumber\\
     &\geq \frac{1}{\sqrt{h}} \int_{\R^d} \int_{\R^d} G_{h}(x,y) |u_{h}(x) - u_{h}(y)| \, dy \,dx \nonumber\\
     &= \int_{\R^d} H_h u_h \, dx. \label{eq:lima}
  \end{align}
  As $u_{h, \eps, R}(\cdot, t=0) = u_0 = u_h(\cdot,t =0)$ and $u_0$ is constant outside $B_L(0)$ we have the following, only on $h$ dependent bound of the integrand
  \begin{align*}
  |H_{h, \eps, R}(x, t = 0) | &\leq \frac{1}{\sqrt{h}} \int_{\R^d} G_h(x,y) |\psi_\eps(u_0(x) - u_0(y))|\,dy \\
  &\leq \frac{1}{\sqrt{h}} \int_{\R^d} G_h(x,y) \mathbbold{1}_{\{x\in B_L(0) \vee y \in B_L(0)\}} \,dy \\
  &=: g(x).
  \end{align*}
  As
  \begin{equation*}
    \int_{\R^d} g(x) \,dx \leq \frac{2}{\sqrt{h}} \int_{B_L(0)} \int_{\R^d} G_h(x,y)  \,dy \, dx \leq \frac{2}{\sqrt{h}} |B_L(0)|,
  \end{equation*} 
  we have that $g$ is integrable. Thus, the dominated convergence theorem yields
  \begin{equation}\label{eq:limb}
    \lim_{R \uparrow \infty} \lim_{\eps \downarrow 0}  \int_{\R^d} |H_{h, \eps, R}(x, t = 0) |\,dx = \int_{\R^d} |H_{h}(x, t= 0) |\,dx.
  \end{equation}
  By Lemma~\ref{lem:uniform_convergenceReps} we also have the uniform convergence on $\R^d \times [0, T]$
  \begin{equation}\label{eq:limc}
    \lim_{h \downarrow 0} \lim_{R \uparrow \infty} \lim_{\eps \downarrow 0} u_{h,\eps,R} = u.
  \end{equation}
  Hence, there exists for every $h \in (0,1]$ both $\eps_h \in (0,1]$ and $R_h > 1$ such that by~\eqref{eq:lima}
  \begin{equation}\label{eq:delta_2_bound}
     \int_{\R^d} H_h u_h \, dx \leq \int_{\R^d} H_{h, \eps_h, R_h}\, u_{h,\eps_h, R_h} \,dx + \delta
  \end{equation}
  and by~\eqref{eq:limb}
  \begin{equation}\label{eq:init_bound_eRh}
    \sup_{h, \eps_h, R_h} \int_{\R^d} |H_{h, \eps_h, R_h}(x, t = 0) |\,dx \leq \sup_{h}  \int_{\R^d} |H_{h}(x, t= 0) |\,dx + \delta
  \end{equation}
  and by~\eqref{eq:limc}
  \begin{equation}\label{eq:limd}
    \lim_{h \downarrow 0} u_{h,\eps_h, R_h} = u \quad \quad \text{ uniformly in } \R^d \times [0,T].
  \end{equation}
  %But as $R_h > K$ and due to our choice of $K$ it even holds
  %\begin{align}
  %  &\left|\int_{\R^d} H_{h, \eps_h, R_h}\, u_{h,\eps_h, R_h} \,dx - \int_{\R^d} \int_{\R^d} G_{h,K}(x,y) \psi_\eps(u_{h,\eps_h, R_h}(x) - u_{h,\eps_h, R_h}(y)) \, dy \,dx\right|\nonumber\\
  %  &\leq  \|u_{h, \eps_h, R_h}\|_\infty \int_{|x| \leq L} \int_{K \leq  z \leq  R_h} G_h(0,z) \,dz \, dx \nonumber\\
  %  &\leq \delta_1. \label{eq:delta_1_bound}
  %\end{align}
  %Thus, the upper bound~\eqref{eq:delta_2_bound} turns into
  %\begin{equation*}
  %  \int_{\R^d} H_h u_h \, dx \leq \int_{\R^d} \frac{1}{\sqrt{h}}\int_{\R^d} G_{h,K}(x,y) \psi_\eps(u_{h,\eps_h, R_h}(x) - u_{h,\eps_h, R_h}(y)) \, dy \,dx + \delta_1 + \delta_2.
  %\end{equation*}
  %With a similar computation as in~\eqref{eq:delta_1_bound} it also holds
  %\begin{equation*}
  %  \sup_{h, \eps_h, R_h} \int_{\R^d} \left|\frac{1}{\sqrt{h}}\int_{\R^d} G_{h,R_h}(x,y) \psi_\eps(u_{h,\eps_h, R_h}(x,0) - u_{h,\eps_h, R_h}(y,0)) \, dy  \right| \, dx \leq \sup_{h}\int_{\R^d} |H_{h}(x,t=0) |\,dx + \delta_2 + \frac{\delta_1}{\|u_0\|_\infty}. 
  %\end{equation*}
  The bound~\eqref{eq:init_bound_eRh}, combined with the properties~\eqref{def:well_prepared_init} of the initial condition, allows applying Lemma~\ref{lem:limit_of_Hh} to deduce the existence of $\sigma$ with $|\sigma|\leq 2c_0$ such that
  \begin{equation}\label{eq:lime}
  H_{h, \eps_h, R_h} \, dx \, dt \xrightharpoonup{*} \nabla \cdot \sigma \,dx\,dt  \quad \quad \text{ weakly-$*$ as measures as } h \downarrow 0.
  \end{equation} 
  Combing the weak-$*$ convergence~\eqref{eq:lime} of $H_{h, \eps_h, R_h}$ with the uniform convergence~\eqref{eq:limd} of $u_{h,\eps_h, R_h}$ lets us pass to the limit to obtain:
  \begin{align*}
    \limsup_{h \downarrow 0} \int_0^T E_h(u_h) \, dt &= \limsup_{h \downarrow 0} \int_0^T \int u_h H_h \,dx \, dt\\
    &\leq \limsup_{h \downarrow 0} \int_0^T \int u_{h, \eps_h, R_h} H_{h, \eps_h, R_h} \,dx \, dt + T\delta\\
    &= \int_0^T \int u \nabla \cdot \sigma \, dx\,dt + T\delta \\
    &= - \int_0^T \int \nabla u \cdot \sigma \, dx\,dt + T\delta \\
    &\leq 2c_0 \int_0^T \int |\nabla u| \, dx \,dt + T\delta .
  \end{align*}
  Taking the limit $\delta \downarrow 0$ yields the desired upper bound
  \begin{equation*}
    \limsup_{h \downarrow 0} \int_0^T E_h(u_h) \, dt \leq 2c_0 \int_0^T \int |\nabla u| \,dt.
  \end{equation*}
  For the lower bound fix first $z$ and take a smooth test function $\xi$ with $|\xi| \leq 1$ to compute 
  \begin{equation*}
    \int |u_h(x) - u_h(x - \sqrt{h}z)| \,dx \geq  \int (u_h(x) - u_h(x - \sqrt{h}z)) \xi(x) \,dx = \int u_h(x) (\xi(x+ \sqrt{h}z) - \xi(x)) \, dx.
  \end{equation*}
  Hence, it holds
\begin{align*}
   \liminf_{h \downarrow 0} \int |u_h(x) - u_h(x - \sqrt{h}z)| \,dx \geq \sup_{|\xi|\leq 1} \int u(x) \nabla \xi(x) \cdot z = \int |\nu(x) \cdot z| |\nabla u|.
\end{align*}
It follows by Fatou's Lemma and using the normal $\nu(x) = \frac{\nabla u(x)}{|\nabla u(x)|}$, where it exists, that
\begin{align*}
  \liminf_{h \downarrow 0} \int_0^T E_h(u_h) \, dt &= \liminf_{h \downarrow 0} \int_0^T \int_{\R^d} G_1(z) \int |u_h(x) - u_h(x - \sqrt{h}z)| \,dx \,dz \, dt \\
  &\geq \int_0^T \int_{\R^d} \int_{\R^d} G_1(z) |\nu(x) \cdot z| \, dz |\nabla u| \, dx \, dt.
\end{align*}
The result follows now as 
\begin{equation*}
  \int_{\R^d} G_1(z) |\nu(x) \cdot z| \, dz = 2c_0
\end{equation*}
for any $x$.
\end{proof}

\section{Properties of Median Schemes}
\label{sec:properties}
In this section we prove several useful properties of the generalized median scheme introduced earlier in~\eqref{def:generlized_scheme}. We will also investigate the connection between the different schemes $u_{h,\eps,R}$ of \eqref{def:u_heps}, $u_{h,R}$ of~\eqref{def:uhR} and $u_h$ of~\eqref{def:semi_discrete_scheme}.

We first prove uniqueness of the generalized median operator.
\begin{lemma}[General Uniqueness of the Median]\label{lem:uniqueness_median}
  Let $f: \R^d \to \R^d$ by continuous and bounded. There exists a unique generalized median and it satisfies the following two Euler-Lagrange inequalities
  \begin{multline}\label{eq:ELI1}
    - \int_{\R^d} K(x,y) \psi(\med_{K,\Psi} f(x) - f(y))\mathbbold{1}_{\{\med_{K,\Psi} f(x) < f(y)\}} \, dy  \\
     \leq \int_{\R^d} K(x,y) \psi^*(\med_{K,\Psi} f(x) - f(y))\mathbbold{1}_{\{\med_{K, \Psi} f(x) \geq f(y)\}} \, dy, 
     \end{multline}
     and 
     \begin{multline} \label{eq:ELI2}
           - \int_{\R^d} K(x,y) \psi_*(\med_{K,\Psi} f(x) - f(y))\mathbbold{1}_{\{\med_{K,\Psi} f(x) \leq f(y)\}} \, dy\\
    \geq \int_{\R^d} K(x,y) \psi(\med_{K,\Psi} f(x) - f(y))\mathbbold{1}_{\{\med_{K,\Psi} f(x) > f(y)\}} \, dy.
     \end{multline}
  If $f_1, f_2: \R^d \rightarrow \R$ are continuous and bounded satisfying $f_1 \leq f_2$ then it holds
  \begin{equation}\label{eq:monotonicity_generalized_median}
    \med_{K, \Psi} f_1 \leq \med_{K, \Psi} f_2.
  \end{equation}
\end{lemma}
\begin{proof}
  We start with existence of a generalized median. As $\Psi$ is convex and $\psi$ is antisymmetric and $\psi \neq \{0\}$ on $\R\setminus 0$ it holds that $\Psi$ is minimal at $0$. Thus,
  \begin{equation*}
    F(c) := \int_{\R^d} K(x, y) \Psi(c-f(y)) \, dy
  \end{equation*}
  is also convex and as $f$ is continuous and $K(z)\,dz < \infty$ it also holds
  \begin{equation*}
    \lim_{|c|\rightarrow \infty} F(c) = \infty.
  \end{equation*}
  So $F(c)$ is convex and coercive implying there exists a minimizer.

  Next, we prove the Euler-Lagrange inequalities. As $F$ is convex we use the subdifferential for the optimality condition 
  \begin{equation*}
    0\in \partial F(c) = \partial\left(\int_{\R^d} K(x, y)  \Psi(\cdot - f(y)) \,dy\right)(c).
  \end{equation*}
  By differentiability of $\Phi$ except at $0$ we have
  \begin{equation*}
    \partial \Psi(c) = \Psi'(c) = \psi(c) \quad \text{ on } \R \setminus\{0\} \quad \quad \text{ and }  \quad \quad \partial \Psi(0) = [\lim_{c \rightarrow 0^-} \psi(c), \lim_{c \rightarrow 0^+} \psi(c)].
  \end{equation*}
  Hence, as $0 \in \partial F(\med_{K, \Psi} f(x))$ we have
  \begin{equation*}
    \int_{\R^d} K(x, y) \psi_*(\med_{K, \Psi} f(x)-f(y)) \,dy \leq 0 \leq \int_{\R^d} K(x, y) \psi^*(\med_{K, \Psi} f(x)-f(y)) \, dy.
  \end{equation*}
  Using $\psi_* = \psi = \psi^*$ on $\R \setminus \{0\}$ yields the Euler-Lagrange inequalities~\eqref{eq:ELI1} and~\eqref{eq:ELI2}.

  Next we prove the monotonicity~\eqref{eq:monotonicity_generalized_median}. Assume~\eqref{eq:monotonicity_generalized_median} is false. Then there exist continuous continuous, bounded functions $f_1 \geq f_2$ and a point $x \in \R^d$ wih two generalized medians $c_1:= \med_{K, \Psi} f_1(x)$ and $c_2:= \med_{K, \Psi} f_2(x)$ with $c_1 < c_2$. One easily checks that $c_i \in [\inf_{y \in \supp K(x, \cdot)} f_i(y), \sup_{y \in \supp K(x, \cdot)} f_i(y)]$ for $i =1,2$. If $c_1 = \inf_{y \in \supp K(x,\cdot)} f_1(y)$ then by the definition~\eqref{def:generalized_median} of the generalized median it already has to hold that $f_1$ is constant in $\supp K(x,\cdot)$ and equal to $c_1$. But then it has to hold $c_1 = f_1(x) \geq \sup_{y \in \supp K(x,\cdot)} f_2(y) \geq c_2$, a contradiction to $c_1 < c_2$. So, by continuity of $f_1$ there must exist $y^* \in \supp K(x, \cdot)$ such that $c_1 < f_1(y^*) < c_2$. Because of $f_1 \geq f_2$ it also holds $c_2 > f_1(y^*) \geq f_2(y^*)$. Now, by continuity of $f_1$ and $f_2$ there exists an $r > 0$ such that
  \begin{equation*}
    B_r(y^*) \subset \{y: c_1 < f_1(y)\} \cap \{y: c_2 > f_2(y)\}.
  \end{equation*} 
 Applying the Euler-Lagrange inequality~\eqref{eq:ELI1} to $c_1$ and~\eqref{eq:ELI2} to $c_2$ as well as using the monotonicity of $\psi$ yields the contradiction
  \begin{align*}
    &- \int_{\R^d} K(x,y) \psi(c_1 - f_1(y))\mathbbold{1}_{\{c_1< f_1(y)\}} \, dy\\
     &\stackrel{\eqref{eq:ELI1}}{\leq} \int_{\R^d} K(x,y) \psi^*(c_1 - f_1(y))\mathbbold{1}_{\{c_1\geq f_1(y)\}} \, dy\\
    &\stackrel{\stackrel{c_1 < c_2}{ f_1 \geq f_2}}{\leq} \int_{\R^d} K(x,y) \psi(c_2 - f_2(y))\mathbbold{1}_{\{c_2 > f_2(y)\}} \, dy - \int_{\R^d} K(x,y) \psi^*(c_1 - f_1(y))\mathbbold{1}_{\{c_1 < f_1(y) \wedge c_2 > f_2(y)\}} \, dy\\
     &\stackrel{\eqref{eq:ELI2}}{\leq} -\int_{\R^d} K(x,y) \psi_*(c_2 - f_2(y))\mathbbold{1}_{\{c_2 \leq f_2(y)\}} \, dy - \int_{B_r(y^*)} K(x,y) \psi^*(c_1 - f_1(y)) \, dy\\
     &\stackrel{c_1 < c_2, f_1 \geq f_2}{<} -\int_{\R^d} K(x,y) \psi(c_1 - f_1(y))\mathbbold{1}_{\{c_1 < f_1(y)\}} \, dy.
  \end{align*}
  We conclude that the generalized median is a monotone operator. This immediately yields the uniqueness of the generalized median which finishes this proof.
\end{proof}

Next we prove that the generalized median scheme is a contracting operator. By an observation of~\cite[Proposition 2]{zbMATH03703495}, this is equivalent to the monotonicity~\eqref{eq:monotonicity_generalized_median} of the median operator of Lemma~\ref{lem:uniqueness_median} if 
\begin{equation}\label{eq:homogenity}
  \med_{K, \Psi} (f + c) = \med_{K, \Psi} (f) + c
\end{equation}
for any continuous function $f$ and constant $c$.
\begin{lemma}[Contraction of Median Operator]\label{lem:contraction}
  The generalized median operator~\eqref{def:generalized_median} satisfies the following contraction property for any continuous, bounded functions $f,g$
  \begin{equation}\label{eq:contraction}
    |\med_{K, \Psi} f(x) - \med_{K, \Psi} g(x)| \leq \max_{y \in \supp K(x, \cdot)} |f(y) - g(y)|.
  \end{equation}
  Furthermore, it holds for any $x,y \in \R^d$ and $f \in C^{0,1}(\R^d)$
  \begin{equation}\label{eq:lipschitz_mon}
  |\med_{K, \Psi} f(x) - \med_{K, \Psi} f(y)| \leq \Lip(f) |x-y|.
  \end{equation}
\end{lemma}
  \begin{proof}
     One easily checks~\eqref{eq:homogenity}. For any $x \in \R^d$ we simply have $f \leq g + \|f-g\|_{L^\infty(\supp(K(x, \cdot)))}$ on $\supp(K(x, \cdot))$, which immediately yields by~\eqref{eq:monotonicity_generalized_median} and~\eqref{eq:homogenity} that
  \begin{equation*}
    \med_{K, \Psi} f(x) \leq \med_{K, \Psi} (g + \|f-g\|_{L^\infty(\supp(K(x, \cdot)))})(x) = \med_{K, \Psi} g(x) + \|f-g\|_{L^\infty(\supp(K(x, \cdot)))}.
  \end{equation*}
  Changing the roles of $f$ and $g$ yields the other inequality which proves~\eqref{eq:contraction}.

It remains to prove~\eqref{eq:lipschitz_mon}. Let $x,y \in \R^d$ be fixed. Define $g(z) = f(z +(y-x))$ such that it holds
$g(x) = f(y)$. Then it holds
\begin{align*}
  \med_{K, \Psi} g(x) &\in \argmin_{c \in \R} \int_{\R^d} K(x-z) \Psi(c - g(z))\,dz \\
  &= \argmin_{c \in \R} \int_{\R^d} K(x-z) \Psi(c - f(z +(y-x)))\,dz \\
  &=\argmin_{c \in \R} \int_{\R^d} K(y-\tilde{z}) \Psi(c - f(\tilde{z}))\,d\tilde{z},
\end{align*}
which implies $\med_{K, \Psi} g(x) = \med_{K, \Psi} f(y)$ by the uniqueness of the generalized median of Lemma~\ref{lem:uniqueness_median}. Consequently, applying~\eqref{eq:contraction} to the pair $f,g$ lets us deduce
\begin{align*}
  |\med_{K, \Psi} f(x) - \med_{K, \Psi} f(y)| &= |\med_{K, \Psi} f(x) - \med_{K, \Psi} g(x)|
  \\& \leq \|f - f(\cdot + (y-x))\|_\infty \leq \Lip(f) |x-y|,
\end{align*}
which is exactly the claimed inequality~\eqref{eq:lipschitz_mon}.
\end{proof}

 The next lemma states the stability of the generalized median operator when cutting off the kernel.
 \begin{lemma}\label{lem:cuttable}
  For any kernel $K$ satisfying~\eqref{def:K_kernel}
  define the cut-off kernel
  \begin{equation*}
    K_R(x,y) = \mathbbold{1}_{|x-y|<R} K(x,y).
  \end{equation*}
  Then, for sufficiently large $R > 0$, any Lipschitz continuous functions $f$ and any $x \in \R^d$ it holds
  \begin{equation}\label{eq:cuttable}
    |\med_{K_R, \Psi} f(x) - \med_{K, \Psi} f(x)| \leq 2\Lip(f)\left( C(R) \right)^{\frac{1}{2d}}.
  \end{equation}
  Here $C(R)$ is a constant only depending on $R$ and the kernel $K$ given by
  \begin{equation*}
    C(R) := \int_{\{z > R\}} K(0,z)\, dz,
  \end{equation*}
  which in particular satisfies
  \begin{equation*}
    C(R) \to 0 \quad\quad \text{ as } R \uparrow \infty.
  \end{equation*}
 \end{lemma}
 \begin{proof}
  If $K$ has compact support there is nothing to prove. So by~\eqref{def:K_kernel} we know
  \begin{equation}\label{eq:inf_kernel_pos}
    \inf_{|z| \in B_R(0)} K(z) > 0 \quad\quad\text{ for any } R > 0.
  \end{equation}
  We argue by contradiction: Assume that~\eqref{eq:cuttable} is false, i.e., there exist $f$ and $x \in \R^d$ such that for every $R > 0$ and $h > 0$, w.l.o.g. (the other case is treated similarly), 
  \begin{equation}\label{eq:not_cuttable}
    \med_{K, \Psi} f(x)  - \med_{K_R, \Psi} f(x) >  2\Lip(f) \left( C(R) \right)^{\frac{1}{2d}} =: \tilde{C} > 0.
  \end{equation} 
  As $f$ is continuous there exists a $z_0 \in \R^d$ such that $\med_{K, \Psi} f(x) = f(x + z_0)$. Let $R_0> 0$ be the fixed radius with
  \begin{equation*}
    \int_{\{|z| > R_0\}} K(0,z) \,dz = \frac{\int_{\R^d}K(0,z)\, dz}{2}.
  \end{equation*}
  Then we can assume that $|z_0| \leq R_0$ as otherwise the Euler-Lagrange inequality~\eqref{eq:ELI1} would be violated. Similarly, there is $|z_R| \leq R_0$ (with the same constant $R_0$ as $K_R$ has its mass more centered than $K$) such that $\med_{K, \Psi} f(x) - \tilde{C} > \med_{K_R, \Psi} f(x) = f(x + z_R)$. By continuity of $f$ there is a $z^*$ on the segment $[z_0, z_R] \subset B_{C\sqrt{h}}(0)$ such that $f(y^*) = \med_{K, \Psi} f(x) - \frac{\tilde{C}}{2}$ for $y^*=x + z^*$. Now define the radius $r = \frac{\tilde{C}}{2 \Lip(f)}= \left(C(R) \right)^{\frac{1}{2d}}$. Then it holds 
  \begin{equation}\label{eq:ball_eps_arg}
     B_r(y^*) \subset \{y: \med_{K_R, \Psi} f(x) < f(y) < \med_{K, \Psi} f(x)\}.
  \end{equation}
  Therefore, by assumption, we have
  \begin{equation} \label{eq:f_g_prop_eps}
    \{y: \med_{K_R, \Psi} f(x) \geq f(y) \} \subset \{y: \med_{K, \Psi} f(x) \geq f(y)\},
  \end{equation}
  and by~\eqref{eq:ball_eps_arg}
  \begin{equation}\label{eq:ball_prop_eps}
    B_r(y^*) \subset \{y: \med_{K, \Psi}f(y) \geq f(y)\}\setminus \{y: \med_{K_R, \Psi}f(y) \geq f(y)\}.
  \end{equation}
  Hence, our construction~\eqref{eq:ball_prop_eps} and the Euler-Lagrange inequality~\eqref{eq:ELI1} for the cut-off kernel imply
  \begin{align}
    -\int_{\R^d} &K_R(x,y) \psi(\med_{K_R, \Psi} f(x) - f(y)) \mathbbold{1}_{\{\med_{K_R, \Psi} f(x) < f(y)\}} \,dy \nonumber\\
    &\stackrel{\eqref{eq:ELI1}}{\leq} \int_{\R^d} K_R (x,y) \psi^*( \med_{K_R, \Psi} f(x) - f(y) )\mathbbold{1}_{\{\med_{K_R, \Psi} f(x) \geq f(y)\}} \, dy \nonumber\\
    & \stackrel{K_R \leq K}{\leq}\int_{\R^d} K (x,y) \psi^*( \med_{K_R, \Psi} f(x) - f(y) )\mathbbold{1}_{\{\med_{K_R, \Psi} f(x) \geq f(y)\}} \, dy\nonumber\\
    & \stackrel{\eqref{eq:ball_prop_eps}, \eqref{eq:f_g_prop_eps}}{\leq} \int_{\R^d} K (x,y) \psi( \med_{K, \Psi} f(x) - f(y) )\mathbbold{1}_{\{\med_{K, \Psi} f(x) > f(y)\}} \, dy- \int_{B_r(y^*)} K(x,y) \, dy . \label{eq:cont_up_to_constant}
  \end{align}
  As by construction $|z^*| \leq R_0$ and $r \leq 1$ we can bound the good, negative term from below by 
  \begin{align*}
  \int_{B_r(y^*)} K(x,y) \, dy &\geq |B_r(y^*)| \inf_{y \in B_r(y^*)} K(x,y)\\
  &\geq r^{d} \omega_d \inf_{z \in B_r(z^*)} K(0,z) \\
  &\geq \sqrt{C(R)} \omega_d \inf_{z \in B_{R_0+1}(0)} K(0,z).
  \end{align*}
  The infimum is positive due to~\eqref{eq:inf_kernel_pos}; thus, for $R$ big enough it holds $ \omega_d\inf_{z \in B_{R_0+1}(0)} K(0,z) \geq 2\sqrt{C(R)}$ and in this case
  \begin{equation*}
    \int_{B_r(y^*)} K(x,y) \, dy \geq 2C(R).
  \end{equation*}
  Plugging this into~\eqref{eq:cont_up_to_constant}, using the Euler-Lagrange inequality~\eqref{eq:ELI2} and the definition of $C(R)$ with $\psi^* \leq 1$ yields the contradiction
  \begin{align*}
    -&\int_{\R^d} K_R(x,y) \psi(\med_{K_R, \Psi} f(x) - f(y)) \mathbbold{1}_{\{\med_{K_R, \Psi} f(x) < f(y)\}} \,dy \\
    &\stackrel{\eqref{eq:ELI2}, \eqref{eq:cont_up_to_constant}}{\leq} -\int_{\R^d} K (x,y) \psi_*( \med_{K, \Psi} f(x) - f(y) )\mathbbold{1}_{\{\med_{K, \Psi} f(x) \leq f(y)\}} \, dy - 2C(R)\\
    &\stackrel{\eqref{eq:not_cuttable}, \eqref{eq:f_g_prop_eps}}{\leq} -\int_{\R^d} K (x,y) \psi( \med_{K_R, \Psi} f(x) - f(y) )\mathbbold{1}_{\{\med_{K_R, \Psi} f(x) < f(y)\}} \, dy - 2C(R)\\
    &\stackrel{|\psi| \leq 1}{\leq} -\int_{\R^d} K_R (x,y) \psi( \med_{K_R, \Psi} f(x) - f(y) )\mathbbold{1}_{\{\med_{K_R, \Psi} f(x) < f(y)\}} \, dy + \int_{B_R(0)^c}\!\! K(z) \, dz - 2C(R)\\
    &\stackrel{\text{ def. } C(R)}{<} -\int_{\R^d} K_R (x,y) \psi( \med_{K_R, \Psi} f(x) - f(y) )\mathbbold{1}_{\{\med_{K_R, \Psi} f(x) < f(y)\}} \, dy.
  \end{align*}
  This proves the claimed~\eqref{eq:cuttable} and finishes the proof.
 \end{proof}

 Let us briefly recall the median schemes from Section~\ref{sec:conv_energie}. First, there is $u_h$ which is the continuous in time median scheme~\eqref{def:semi_discrete_scheme} with kernel $K = G_h$ and $\Psi = |\cdot|$. Second, $u_{h,R}$ is the scheme~\eqref{def:uhR} with $K(x-y) = \mathbbold{1}_{|x-y| \leq R}G_h(x,y)$ and $\Psi = |\cdot|$. Third, we used $u_{h,\eps,R}$ given by~\eqref{def:u_heps} with $K = \mathbbold{1}_{|x-y| \leq R}G_h(x,y)$ and differentiable $\Psi = \Psi_\eps$ given by~\eqref{def:psi_eps}. At last, there is also $u$, the unique viscosity solution to the level-set mean curvature flow.
 
 We have already proven in Proposition~\ref{the:viscosity} that 
 \begin{equation*}
  u_{h} \rightarrow u
 \end{equation*}
 uniformly in $\R^d \times [0,T]$ for any $T > 0$. We will now also prove that $u_{h, \eps, R}$ converges uniformly to $u_{h, R}$ and this in turn converges to $u_h$:
 
\begin{lemma}\label{lem:uniform_convergenceReps}
  We have the two uniform convergences on $\R^d \times [0, T]$ for any $0 \leq T < \infty$:
  \begin{align}
    \lim_{R \uparrow \infty} u_{h,R} &= u_h, \label{eq:convergence_uhR}\\
    \lim_{\eps \downarrow 0} u_{h,\eps, R} &= u_{h,R}\label{eq:convergence_uhepsR}.
  \end{align}
\end{lemma}
\begin{proof}
  \emph{Step 1 -- $C^{0,1}$-contraction:} Let us first prove the following useful claim:  
\emph{All median schemes are contractions in the sense that the $C^{0,1}$ norm is non-increasing, i.e, for any $0\leq t_1 \leq t_2 < \infty$ let $v$ represent any of the functions $u_h, u_{h, R}$ or $u_{h, \eps, R}$ then it holds
    \begin{align}
      \|v(\cdot, t_1)\|_{C^0} &\geq \|v(\cdot, t_2)\|_{C^0}\label{eq:max_non_increasing} \\
      \Lip(v(\cdot, t_1))&\geq \Lip(v(\cdot, t_2)) \label{eq:Lip_non_increasing}.
    \end{align}}

  Indeed, \eqref{eq:max_non_increasing} follows immediately from the comparison principle of Lemma~\ref{lem:comp_princ} when comparing $v$ to the constant functions with values $\min_{y \in \R^d} v(y,t_1)$ and $\max_{y \in \R^d} v(y,t_1)$. Also for  \label{eq:Lip_non_increasing} one can use the comparison principle of Lemma~\ref{lem:comp_princ}; this time between $v$ and the function $ z \mapsto v(z + x - y, t_1) \pm \Lip(v) |x-y|$.

\emph{Step 2 -- Argument for $\lim_{R\uparrow \infty} u_{h, R} = u_h$:}
  By Lemma~\ref{lem:cuttable} it holds: For any bounded, Lipschitz continuous function $f$, any $x \in \R^d$, and sufficiently large $R > 0$ it holds
  \begin{equation*}
    |\med_{h,R} f(x) - \med_h f(x)| \leq 2\Lip(f)\left( C(R,h) \right)^{\frac{1}{2d}}.
  \end{equation*}
  Here $C(R,h)$ is a constant only depending on $R, h$ given by
  \begin{equation*}
    C(R,h) := \int_{\{|z| > R\}} G_h(0,z)\, dz,
  \end{equation*}
  which in particular satisfies
  \begin{equation*}
    C(R,h) \to 0 \quad\quad \text{ as } R \uparrow \infty.
  \end{equation*}
  Thus, comparing $u_h$ to $v(x,t) = u_{h,R}(x,t) \pm \frac{2t}{h} \Lip(u_0) C(R)^{\frac{1}{2d}}$ in Lemma~\ref{lem:comp_princ} yields 
  \begin{equation*}
    u_h(x,t) - u_{h,R} \lessgtr \pm \frac{2t\Lip(u_0) C(R, h)^{\frac{1}{2d}}}{h}.
  \end{equation*}
  This means that $u_{h,R}$ converges uniformly to $u_h$ as $R \uparrow \infty$ on any compact time interval, because $h$ is fixed and $C(R, h) \to 0$ as $R \uparrow \infty$.

    \emph{Step 3 -- Argument for $\lim_{\eps \downarrow 0} u_{h,\eps,R} = u_{h, R}$:}
  We will first prove the following claim: For any continuous function $f$ it holds 
\begin{equation}\label{eq:eps_bound_epsmed}
  |\med_{h,R} f(x)  - \med_{h, \eps, R} f(x)| \leq \eps \quad \quad \text{ for all } x \in \R^d.
\end{equation}
Once ~\eqref{eq:eps_bound_epsmed} is shown, we can infer~\eqref{eq:convergence_uhepsR}. Indeed, comparing $u_{h,R}$ to $v(x,t) = u_{h, \eps, R}(x,t) \pm \frac{t}{h}\eps$ in Lemma~\ref{lem:comp_princ} yield
\begin{equation*}
  u_{h,R} - u_{h, \eps, R} \lessgtr \pm \frac{t}{h}\eps.
\end{equation*}

To prove the claimed~\eqref{eq:eps_bound_epsmed}, assume the contrary is true. So there exists a fixed $x$ such that $\med_{h, \eps, R} f (x) > \med_{h,R} f(x) + \varepsilon$ (the other case is handled analogously).
Note, by the properties~\eqref{def:psi_eps} of $\psi_\eps$ it holds $\psi_\eps(\med_{h, \eps, R} f(x) - f(y)) = 1$ in the set $\{y: \med_{h,R} f(x) \geq f(y) \}$. Hence, using the Euler-Lagrange inequality~\eqref{eq:ELI1} we have
\begin{align}
  \int_{\R^d} G_{h, R}(x,y) \mathbbold{1}_{\{\med_{h,R} f(x) < f(y) \}} \, dy &\leq \int_{\R^d} G_{h, R}(x,y) \mathbbold{1}_{\{\med_{h,R} f(x) \geq f(y) \}} \, dy \nonumber\\
  &= \int_{\R^d} G_{h, R}(x,y) \mathbbold{1}_{\{\med_{h,R} f(x) \geq f(y) \}} \psi_\eps(\med_{h, \eps, R} f(x) - f(y)) \, dy.\label{eq:R_prop}
\end{align}
As $f$ is continuous, we may make a similar construction as in the proof of Lemma~\ref{lem:contraction} (see~\eqref{eq:ball_prop_eps} for more details) to obtain a point $y^* \in \R^d$ and $r > 0$ satisfying
      \begin{equation*}
      B_r(y^*) \subset \{y: \med_{h, \eps, R} f(x) > f(y)\} \cap \{y: \med_{h,R} f(x) \geq f(y) \}^c.
    \end{equation*}
Thus, we get the strict inequality 
\begin{multline*}
  \int_{\R^d} G_{h, R}(x,y) \mathbbold{1}_{\{\med_{h,R} f(x) \geq f(y) \}}  \psi_\eps(\med_{h, \eps, R} f(x) - f(y)) \, dy\\
  <  \int_{\R^d} G_{h, R}(x,y) \mathbbold{1}_{\{\med_{h, \eps, R} f(x) > f(y) \}} \psi_\eps(\med_{h, \eps, R} f(x) - f(y)) \, dy.
\end{multline*}
Plugging this into~\eqref{eq:R_prop} and using $\psi_\eps \geq -1$ as well as $\med_{h, \eps, R} f(x) > \med_{h,R} f(x)$ once more yields the desired contradiction
\begin{align*}
  &\int_{\R^d} G_{h, R}(x,y) \mathbbold{1}_{\{\med_{h,R} f(x) < f(y) \}} \, dy \\
  &<  \int_{\R^d} G_{h, R}(x,y) \mathbbold{1}_{\{\med_{h, \eps, R} f(x) > f(y) \}} \psi_\eps(\med_{h, \eps, R} f(x) - f(y)) \, dy\\
  &\leq -\int_{\R^d} G_{h, R}(x,y) \mathbbold{1}_{\{\med_{h, \eps, R} f(x) \leq f(y) \}} \psi_\eps(\med_{h, \eps, R} f(x) - f(y)) \, dy\\
  &\leq \int_{\R^d} G_{h, R}(x,y) \mathbbold{1}_{\{\med_{h,R} f(x) < f(y) \}} \, dy.
\end{align*}
This proves the claim~\eqref{eq:eps_bound_epsmed} and finishes the proof of this lemma.
\end{proof}

 The following lemma proves that $u_h$ and $u_{h,R}$ stay constant outside a big ball if the initial condition satisfies this. Note, this is not true for $u_{h, \eps,R}$; in this case prove in Lemma~\ref{lem:exp_decay} an exponential decay.

\begin{lemma}\label{lem:const_outside_ball}
  If $u_0 \in C^{0,1}(\R^d)$ is constant outside a ball $B_L(0)$ for some $L > 0$ then it holds that $u_{h}$ and $u_{h,R}$ are also constant on $\big(\R^d \setminus B_L(0) \big)\times [0, \infty)$.  
\end{lemma}
\begin{proof}
  Assume that $u_0$ has value $c \in \R$ outside $B_L(0)$. Define for any $\lambda \in \S^{d-1}$ the two functions
  \begin{equation*}
    \zeta(x) = \Lip(u_0)\max \{0, \lambda \cdot x +L\} + c \quad \quad \text{ and } \quad \quad \xi(x) = - \zeta(x);
  \end{equation*}
  Then it holds 
  \begin{equation*}
    \xi \leq u_0 \leq \zeta
  \end{equation*}
  and $\xi = \zeta = c$ on the half-space $H_{\lambda} := \{x \in \R^d: \lambda \cdot x + L \geq 0\}$.
  As every level-set of $\xi$ and $\zeta$ is either a half-spaces or the full space $\R^d$, the Euler-Lagrange inequalities~\eqref{eq:ELI1} and~\eqref{eq:ELI2} imply that $\xi$ and $\zeta$ are stationary solutions to~\eqref{def:semi_discrete_scheme} and~\eqref{def:uhR}. Hence, the comparison principle of Lemma~\ref{lem:comp_princ} yield for any $(x, t) \in \R^d \times [0, \infty)$ and for $v$ either $u_h$ or $u_{h,R}$
  \begin{equation*}
    \xi(x) \leq v(x,t) \leq \zeta(x).
  \end{equation*}
  In particular, it holds $v(x,t) = c$ for every $(x,t) \in H_{\lambda} \times [0, \infty)$. As $\lambda \in \S^{d-1}$ was arbitrary and as $\bigcup_{\lambda \in \S^{d-1}} H_\lambda = \R^d \setminus B_L(0)$, we conclude the statement of this lemma.
\end{proof}

Finally, we derive exponential decay bounds on $u_{h,\eps, R}$ that are needed in the proof of Theorem~\ref{cor:energy_conv} and Lemma~\ref{lem:mono_energie}.

\begin{lemma}\label{lem:exp_decay}
If $u_0 = c$ outside $B_L(0)$ then $u_{h, \eps, R}$, $\med_{h, \eps, R} u_{h,\eps, R}$ and $H_{h, \eps, R}$ have exponential decay, i.e., for any $t \in [0,T]$ and any $x \in \R^d$
\begin{align}
  |u_{h,\eps, R}(x,t) - c| &\leq C_2 e^{- (|x| - L)}, \label{eq:exp_u}\\
  |H_{h,\eps, R}(x,t) | &\leq C_3(\eps) e^{- (|x| - L)} \label{eq:exp_H},\\
  |\med_{h, \eps, R} u_{h,\eps, R}(x,t) - c| &\leq C_4 e^{- (|x| - L)} \label{eq:exp_medu}.
\end{align}
The constants $C_2$ and $C_4$ do not depend on $\eps$, while $C_3(\eps)$ may depend on $\eps$. They may depend on $h, R$ or $T$.
\end{lemma}
\begin{proof}
W.l.o.g.\ we may assume that $c = 0$. For easier readability we drop the $h$, $\eps$ and $R$ subscripts in $u_{h,\eps, R}$, $H_{h,\eps, R}$ and $\med_{h, \eps, R}$. 

\emph{Step 1 -- Median estimate:} We claim there exists a constant $C$ only depending on $\Lip(u_0)$, $h$ and $R$ but not on $\eps$ such that
\begin{equation}\label{eq:med_u_by_u}
|\med u(x)|  \leq C \left(|u(x)| + \int_{\R^d} G_h(x,y) |u(y)|\,dy\right).
\end{equation}
Indeed, define the function 
\begin{equation*}
  F(m; x):= \int_{\R^d} G_{h,R}(x,y) \Psi_\eps(m- u(y)) \, dy.
\end{equation*}
The Euler-Lagrange equation resulting from the definition of $\med$ reads
\begin{equation}\label{eq:epsEL}
  0 = F'(\med u(x); x).
\end{equation}
Choose $z$ with $\med u(x) = u(z)$. W.l.o.g\ assume that $\med u(x) \geq u(x)$. By the continuity of $u$, for any $m \in [u(x), u(z)]$ there exists a $z_m$ such that $m=u(z_m)$ and $|x-z_m| \leq R$. As $\Lip(u) \leq \Lip(u_0)$ by~\eqref{eq:Lip_non_increasing}, there exists an $r$ independent of $m$, $h$, $R$ and $\eps$, with
\begin{equation*}
  u(y) \in \left[u(z_m)- \frac{\eps}{2}, u(z_m) + \frac{\eps}{2}\right] \quad \quad \text{ for all } y \in B_r(z_m).
\end{equation*}
Combining this with the previous observation $|x-z_m| \leq R$ yields with $\psi_\eps' > c_1/\eps$ on the compact interval $[-\frac{\eps}{2}, \frac{\eps}{2}]$ for all $y \in B_r(z_m)$ that
\begin{equation*}
  G_{h,R}(x,y) \psi''_\eps(m - u(y))\,dy \geq G_{h}(0, e_1 R) \frac{c_1}{\eps}  =: \frac{\lambda}{\eps} > 0.
\end{equation*}
This lets us bound the second derivative of $F(\cdot; x)$ for sufficiently large $R$ and all $m \in [u(x), \med u(x)]$ by
\begin{align*}
  F''(m; x) = \int_{\R^d} G_h(x,y) \psi''_\eps(m - u(y)) \,dy \geq \frac{1}{4}|B_r(z_m)| \frac{\lambda}{\eps} > 0.
\end{align*}
But now by the mean value theorem and~\eqref{eq:epsEL} it follows with
\begin{equation*}
    F'(u(x); x) \leq \int_{\R^d} G_h(x,y) \psi_\eps(u(x) - u(y)) \,dy \leq \frac{C_1}{\eps} \left( |u(x)| + \int_{\R^d} G_h(x,y)|u(y)| \, dy\right),
\end{equation*}
where we used $|\psi_\eps'|\leq C_1/\eps$ and $\psi_\eps(0) = 0$, that
\begin{align*}
  (|\med u(x)| - |u(x)|)\ \frac{\lambda}{4 \eps}|B_r(0)| &\leq |\med u(x) - u(x)|\ \frac{\lambda}{4 \eps}|B_r(0)|\\
  & \leq |F'(u(x) ; x) - F'(\med u(x); x)| \\
  &\leq \frac{C_1}{\eps} \left(|u(x)| + \int_{\R^d} G_h(x,y)|u(y)| \, dy\right).
\end{align*}
Rearranging yields the claim~\eqref{eq:med_u_by_u}.

\emph{Step 2 -- Argument for~\eqref{eq:exp_u}}: Let us introduce the weighted norm 
\begin{equation*}
  \|f\|_{L} := \sup_{x \in \R^d} e^{(|x| - L)_+} |f(x)|.
\end{equation*}
Recall that we assumed $u_0$ is bounded and equal to zero for $|x|>L$ and thus it holds
\begin{equation*}
  \|u_0\|_L \leq \|u_0\|_\infty < \infty.
\end{equation*}
Integrating~\eqref{def:u_heps} from $0$ to $t$ yields
\begin{equation*}
  u(t) = u_0 + \int_0^t \frac{\med u(s) - u(s)}{h} \, ds.
\end{equation*}
As $\|\cdot\|_L$ is a norm we can use the triangle and Jensen's inequality and finally \emph{Step 1} to estimate
\begin{align*}
  \|u(t)\|_L &\leq \|u_0\|_L + \frac{1}{h} \int_0^t \|\med u(s)\|_L + \| u(s)\|_L \,ds \\
  &\leq \|u_0\|_L + \frac{1}{h} \int_0^t C\left\|\int_{\R^d}  G_h(\cdot,y) u(y,s)\,dy + |u(s)| \right\|_L + \| u(s)\|_L \,ds  \\
  &\leq \|u_0\|_L + \frac{1}{h} \int_0^t \int_{\R^d} C G_h(\cdot,y) \|u(s)\|_L\,dy  + (1+C)\| u(s)\|_L \,ds\\
    &\leq \|u_0\|_L + \frac{1}{h} \int_0^t (1+2C)\| u(s)\|_L \,ds.
\end{align*}
Thus, by Gr{\"o}nwall's inequality we conclude
\begin{equation*}
  \|u(t)\|_L \leq e^{t(1+2C)}\|u_0\|_L \lesssim e^{t(1+2C)}.  
\end{equation*}
But now by the definition of $\|\cdot\|_L$ we can finish this step, as
\begin{equation}\label{eq:exp_decay_u}
  |u(x,t)| \leq e^{-(|x| - L)_+} \|u(t)\|_L \lesssim e^{-(|x| - L)_+} e^{t(1+2C)}.
\end{equation}

\emph{Step 3 -- Argument for~\eqref{eq:exp_H}:} Simply using $\psi'_\eps(0) = 0$ and $|\psi''_\eps| \leq C_\eps$ we crudely estimate
\begin{align}
  |H(x)| &\leq \frac{1}{\sqrt{h}} \int_{\R^d} G_h(x,y) |\psi'_\eps(u(x) - u(y))|\,dy \nonumber \\
  &\leq \frac{C_\eps}{\sqrt{h}} \int_{\R^d} G_h(x,y) \big(|u(x)| + |u(y)|\big)\,dy \nonumber\\
  &\leq C(\eps, h)\left( e^{-(|x| - L)_+} + \int_{\R^d} G_h(x,y) e^{-(|x| - L)_+} \,dy\right). \label{eq:H_exp_est}
\end{align}
We use the general inequality
\begin{equation*}
  (|y| - L)_+ \geq (|x| - L)_+ - |x-y| \quad \quad \text{ for any } x,y \in \R^d,
\end{equation*}
to get the control
\begin{align}\label{eq:exp_control_convolution}
  \int_{\R^d} G_h(x,y) e^{-(|y| - L)_+} \,dy &\leq e^{-(|x| - L)_+}\int_{\R^d} G_h(x,y) e^{|x-y|} \,dy \\
  &= e^{-(|x| - L)_+} \int_{\R^d} G_h(0,z) e^{|z|} \,dz \lesssim  e^{-(|x| - L)_+} \nonumber.
\end{align}
Plugging this into~\eqref{eq:H_exp_est} yields~\eqref{eq:exp_H}.

\emph{Step 4 -- Argument for~\eqref{eq:exp_medu}:} Combining~\eqref{eq:med_u_by_u}, \eqref{eq:exp_decay_u} and~\eqref{eq:exp_control_convolution} yields
\begin{equation*}
  |\med u (x)| \lesssim |u(x)| + \int_{\R^d} G_h(x,y) |u(y)|\,dy \lesssim e^{-(|x| - L)_+} + \int_{\R^d} G_h(x,y) e^{-(|y| - L)_+}\,dy \lesssim e^{-(|x| - L)_+}.
\end{equation*}
This finishes the proof.
\end{proof}

\section{Construction of Initial Conditions}
\label{sec:initial}

In this section we prove that our main result, Theorem~\ref{cor:energy_conv}, is not an empty statement, in the sense that there exists well prepared initial conditions that satisfy the condition~\eqref{def:well_prepared_init}. In general we expect that a broad class of initial conditions satisfy~\eqref{def:well_prepared_init}.
\begin{lemma}
  Let $\Omega_0 \subset \R^d$ be a bounded set with $C^3$ boundary. Then there exists a function $u_0 \in C^3_b(\R^d\setminus S) \cap C^{0,1}(\R^d)$ where $S$ is closed and satisfies $\overline{\R^\setminus S} = \R^d$ and a scalar $R > 0$ with 
  \begin{align}
    \{ u_0 > 0 \} &= \Omega_0, \label{eq:zero_level_of_u_0}\\
    u_0 &\equiv -1 \quad \quad \text{ on } \R^d \setminus B_R(0), \label{eq:const_out_u_0}\\
    \sup_{h \in (0, 1]} \int_{\R^d} \left| \frac{1}{\sqrt{h}}\int_{\R^d} G_h(x,y) \sign(u_0(x) - u_0(y)) \, dy \right| \, dx &< \infty \label{eq:L1_u_0},\\
    \sup_{h \in (0,1]} E_h(u_0) < \infty \label{eq:energy_bound_u_0}.
  \end{align}
\end{lemma}
\begin{proof}
  Denote by $d$ the signed distance function to $\Omega_0$, i.e.,
  \begin{equation*}
    d(x) = \begin{cases}
      -\dist(x, \Omega_0) & \text{ if } x \in \R^d \setminus \Omega_0,\\
      \dist(x, \R^d \setminus \Omega_0) & \text{ if } x \in \Omega_0.
    \end{cases}
  \end{equation*}
  Furthermore, as $\Omega_0$ is bounded we know that the reach $2\delta_0$ of the manifold $\partial \Omega_0$ is positive. Let $\phi \in C^\infty(\R^d)$ be a linear interpolation on $[-\delta_0, \delta_0]$ that is constantly extended outside:
  \begin{equation*}
    \phi(s) = \begin{cases}
      -1 & \text{ if } s \leq \delta_0,\\
      s &\text{ if } s \in [-\delta_0, \delta_0],\\
      1 & \text{ if } s \geq \delta_0.
    \end{cases}
  \end{equation*}
Now we define our initial conditions by
\begin{equation*}
  u_0 = \phi \circ d.
\end{equation*}
By construction, \eqref{eq:zero_level_of_u_0} and~\eqref{eq:const_out_u_0} hold. Also, the regularity of the boundary of $\Omega_0$ implies that $d \in C^3_b(\R^d)$ which means also $u_0 \in C^3_b(\{x: |d(x)| \neq \delta_0\} )$; so setting $S= \{x: |d(x)| = \delta_0\}$ yields the desired regularity properties.

Next we prove~\eqref{eq:L1_u_0}. We begin by showing the convergence
\begin{equation}\label{eq:H_h_to_H}
  \lim_{h \downarrow 0} \frac{1}{\sqrt{h}} \int_{B_{\sqrt{h} R_h}(x)} G_h(x,y) \sign(u_0(x) - u_0(y)) \,dy = c_0 \nabla \cdot \left(\frac{\nabla u_0}{|\nabla u_0}\right),
\end{equation}
if $u_0 \in C^3(B_{\delta_0}(x))$ and $R_h = \sqrt{6|\log (h)|}$ and $|\nabla u_0(x)| > 0$ . Moreover, the rate at which this convergence holds depends on $x$ only through $|\nabla^2 u_0(x)|/|\nabla u_0(x)|$, i.e., it is uniform on sets where the Hessian is bounded from above and the gradient is bounded from below.

Since $\operatorname{sign}\big(u_0(x) - u_0(y)\big)$ is invariant under multiplying $u_0$ by a
positive constant, and since the normalized gradient $\nabla u_0 / |\nabla u_0|$ is invariant
under the same rescaling, both sides of \eqref{eq:H_h_to_H} are unchanged if we replace $u_0$
by
\[
  v := \frac{u_0}{|\nabla u_0(x)|}.
\]
This gives $|\nabla v(x)| = 1$, so it suffices to prove \eqref{eq:H_h_to_H} under this
normalization. The only quantity affected is the Hessian, which rescales as
\[
  \nabla^2 v = \frac{\nabla^2 u_0}{|\nabla u_0(x)|}.
\]

We make a change of variables $z = \frac{y-x}{\sqrt{h}}$ to obtain
\begin{align*}
  \frac{1}{\sqrt{h}} \int_{B_{\sqrt{h}R_h}(x)} G_h(x,y) \sign(u_0(x) - u_0(y)) \,dy &=\frac{1}{\sqrt{h}} \int_{B_{R_h}(x)} G_1(0,z) \sign(u_0(x) - u_0(x + \sqrt{h} z)) \,dz.
\end{align*}

W.l.o.g.\ we can assume that $u_0(x) = 0, \nabla u_0(x) = e_d$. 
We are interested in understanding the behavior of $\sign(u_0(x) - u_0(x + \sqrt{h} z))$. Thus, we need to find all point $z$ at which $u_0(x) - u_0(x + \sqrt{h} z) = 0$. In particular, if we denote $(z', z_d) = z$, then we claim that for $h$ small enough, there is a unique $z_d = g_h(z')$ with $u_0(x) - u_0(x + \sqrt{h} z) = 0$. 
By Taylor expansion, since $u_0 \in C^3(B_{\sqrt{h}R_h}(x))$ (for $h$ small enough), we have
\begin{equation*}
  0 \stackrel{!}{=} u_0(x) - u_0(x+ \sqrt{h}z) = \sqrt{h} z_d + \frac{h}{2} z^\intercal \nabla^2 u_0(x) z + \mathcal{O}(h^{\frac{3}{2}} |z|^3).
\end{equation*}
Rearranging and dividing by $\sqrt{h}$ turns the previous equation to 
\begin{equation}\label{eq:zero_sign_level}
  z_d = -\frac{\sqrt{h}}{2} z^T \nabla^2 u_0(x) z + \mathcal{O}(h |z|^3).
\end{equation}
Using $\|\nabla^2 u_0\|_\infty < \infty$ yields the uniform bound
\begin{equation*}
  z_d = \mathcal{O}(\sqrt{h}|z|^2).
\end{equation*}
So there is a zero of $\sign(u_0(x) - u_0(x + \sqrt{h} z))$ close to $0$; remember that $z_d= 0$ meant $x_d = y_d$ before the coordinate transform. 

Next we want to see that there is no second zero. Therefore, define the function
\begin{equation*}
  F(z_d) := u_0(x + \sqrt{h} z) - u_0(x).
\end{equation*}
Differentiating yields 
\begin{equation*}
  \partial_{z_d} F(z_d) = \partial_{z_d} u(x + \sqrt{h}z) = \partial_{x_d} u_0(x+ \sqrt{h}z) \sqrt{h}.  
\end{equation*}
As $\partial_{x_d} u_0(x) = 1$ and $\|\nabla^2 u_0\|_\infty < \infty$ we have that $\partial_{z_d} F(z_d) > 0$ for $|\sqrt{h} z| < \delta$ and some $\delta > 0$ independent of $x$. In other words, $F(z_d)$ is increasing for $|z| \leq \frac{\delta}{\sqrt{h}}$. Remember we cut of the domain so that we have $|z| \leq R_h = \sqrt{6 |\log h|}$. This means for $h$ small enough it holds that $F(z_d)$ is increasing for all $z$ in our cut-off domain. Hence, $F$ can only have one zero in our domain. 

Furthermore, plugging $z_d = \mathcal{O}(\sqrt{h}|z|^2)$ back into~\eqref{eq:zero_sign_level} and remembering that due to our cut-off it holds $\mathcal{O}(h|z|^4) = \mathcal{O}(h\, (3|\log(h)|)^2) = o(\sqrt{h})$, gives the looked-for characterization of the unique zero
\begin{equation*}
  z_d = -\frac{\sqrt{h}}{2} z'^\intercal \nabla'^2 u_0(x) z' + \mathcal{O}(h |z|^4) = -\frac{\sqrt{h}}{2} z'^\intercal \nabla'^2 u_0(x) z' + o(\sqrt{h}) =: g_h(z').
\end{equation*}

Recall that $z_d = \mathcal{O}(\sqrt{h}|z|^2) = \mathcal{O}(h^{\frac{1}{4}})$. Now using the properties of the heat kernel we have
\begin{align*}
  &\lim_{h \downarrow 0} \frac{1}{\sqrt{h}} \int_{ B_{R_h}(0)} G_1(0,z) \sign(u_0(x) - u_0(x + \sqrt{h} z)) \,dz \\
  &= \lim_{h \downarrow 0} \frac{1}{\sqrt{h}} \int_{B^{d-1}_{R_h - \mathcal{O}(h^{\frac{1}{4}})}(0)}G_1^{d-1}(0,z') \int_{-(R_h -|z'|)}^{R_h -|z'|} G_1^1(0,z_d) \sign(u_0(x) - u_0(x + \sqrt{h} z)) \,d z_d \, dz'.
\end{align*}
As the zero $g_h(z')$ is in $[-(R_h -|z'|), R_h -|z'|]$ we have by symmetry of $G_1^1$
\begin{align*}
  &\int_{-(R_h -|z'|)}^{R_h -|z'|} G_1^1(0,z_d) \sign(u_0(x) - u_0(x + \sqrt{h} z)) \,dz_d\\ 
  &= -\int_{-(R_h - |z'|)}^{g_h(z')} G_1^1(0,z_d) \, d z_d + \int_{g_h(z')}^{R_h - |z'|} G_1^1(0,z_d) \, d z_d\\
  &= -2 \int_{0}^{g_h(z')} G_1^1(0,z_d) \, d z_d\\
  &= -\erf(\frac{g_h(z')}{2}).
\end{align*}
Furthermore, Taylor-expanding the error function gives
\begin{equation*}
  \erf(\frac{g_h(z')}{2}) = \erf(0) + \erf\nolimits'(0) \frac{g_h(z')}{2} + \mathcal{O}(|g_h|^2) = 2c_0 \frac{g_h(z')}{2} + o(\sqrt{h}).
\end{equation*}
Combing the previous computations leaves us with 
\begin{align*}
  &\lim_{h \downarrow 0} \frac{1}{\sqrt{h}} \int_{ B_{R_h}(0)} G_1(0,z) \sign(u_0(x) - u_0(x + \sqrt{h} z)) \,dz \\
  &=\lim_{h \downarrow 0} \frac{1}{\sqrt{h}} \int_{B^{d-1}_{R_h - \mathcal{O}(h^{\frac{1}{4}})}(0)}G_1^{d-1}(0,z') \big(-c_0 g_h(z') + o(\sqrt{h})\big)\, dz'\\
  &= \lim_{h \downarrow 0} \int_{B^{d-1}_{R_h - \mathcal{O}(h^{\frac{1}{4}})}(0)} G_1^{d-1}(0,z') \big( \frac{c_0}{2} z'^\intercal \nabla'^2 u_0(x) z' + o(1)\big)\, dz'\\
  &= \lim_{h \downarrow 0} \frac{c_0}{2} \int_{\R^{d-1}} G_1^{d-1}(0,z')  z'^\intercal \nabla'^2 u_0(x) z'\, dz' + o(1),
\end{align*}
where we used in the last equation the exponential decay of the heat kernel (see~\eqref{eq:cut_off_u_tilde} for more details).
The symmetry of the heat kernel yield
\begin{align*}
    \int_{\R^{d-1}} z_i z_j G_1^{d-1}(0, z') \,dz' = \delta_{ij} \int_{\R} z_1^2 G_1^1(0, z_1) \, d z_1
\end{align*}
and
\begin{align*}
    \int_{\R} - 2 z_1 \partial_{z_1} G_1^1(0,z_1) \, dz_1 = \int_{\R} 2 G_1^1(0, z_1) \, dz_1 = 2
\end{align*}
and hence
\begin{equation*}
    \int_{\R^{d-1}} z_i z_j G_1^{d-1}(0, z') \, dz' = 2 \delta_{ij}.
\end{equation*}
This let's us compute
\begin{align*}
   \int_{\R^{d-1}} G_1^{d-1}(0,z')  z'^\intercal \nabla'^2 u_0(x) z'\, dz' &= \sum_{i = 1}^{d-1} \int_{\R} z_i^2 \partial_{x_i} \partial_{x_j}u_0(x) G_1^{d-1}(0, z') \,dz'\\
   &= 2 (\Delta u_0(x) - \partial_{x_d}\partial_{x_d} u_0(x)).
\end{align*}
Note, as $\nabla u_0(x) = e_d$ we also have
\begin{equation*}
  \Delta u_0(x) -  \partial_{x_d}\partial_{x_d} u_0(x) = \Delta u_0(x) - \frac{\nabla u_0(x)^\intercal \nabla^2 u_0(x) \nabla u_0(x)}{|\nabla u_0(x)|^2} = \nabla \cdot \left(\frac{\nabla u_0(x)}{|\nabla u_0(x)|} \right)|\nabla u_0(x)|.
\end{equation*}
Combing the above observations yields the claimed~\eqref{eq:H_h_to_H}.

Denote by
\begin{align*}
  \Omega_- &:= \{x \in \R^d: d(x) \leq - \delta_0\},\\
  \Omega_+ &:= \{x \in \R^d: d(x) \geq + \delta_0\},\\
  \Omega_{:)} &:= \{x \in \R^d:-\delta_0 < d(x) < \delta_0\},
\end{align*}
the sets of points with vanishing gradient $\Omega_-$ and $\Omega_+$, as well as the set of not-vanishing gradient $\Omega_{:)}$.  We now claim that 
\begin{align}\label{eq:conv_to_per}
  \lim_{h \downarrow 0}\int_{\Omega_- \cup \Omega_+} |H_h(x)|\,dx = \Per(\Omega_-) + \Per(\Omega_+)
\end{align}  
and 
\begin{align}\label{eq:conv_to_H}
    \lim_{h \downarrow 0}  H_h(x) = 2 c_0 \nabla \cdot \left( \frac{\nabla u_0}{|\nabla u_0|} \right) \quad \quad \text{uniformly on $\Omega_{:)}$}.
\end{align}
Showing this will be enough as $\Omega_-$ and $\Omega_+$ have compact $C^3$ boundary by construction and are thus sets of finite perimeter; furthermore, it holds $|\nabla u_0| =1$, $\phi' \circ d = 1$, $\phi''\circ d =0$ on $\Omega_{:)}$ so that
\begin{equation*}
  \nabla \cdot \left( \frac{\nabla u_0}{|\nabla u_0|} \right) = \Delta u_0 = \Delta d \ \phi' \circ \delta + |d|^2 \phi'' \circ d = \Delta d \quad \quad \text{ on } \Omega_{:)}.
\end{equation*}
This yields with a generic constant $C$
\begin{align*}
  \sup_{h \in (0,1]} \int_{\R^d}|H_h| \,dx &\leq C + \Per(\Omega_-) + \Per(\Omega_+) + 2c_0 \int_{\Omega_{:)}}(\nabla \cdot \left( \frac{\nabla u_0}{|\nabla u_0|} \right) + C)\,dx\\
  & \leq C + \Per(\Omega_-) + \Per(\Omega_+)+ 2c_0|\Omega_{:)}|(C + \|\nabla^2 d\|_\infty) \\
  &< \infty.
\end{align*}
 
So, it remains to prove~\eqref{eq:conv_to_per} and~\eqref{eq:conv_to_H}. Let us first prove~\eqref{eq:conv_to_per} and concentrate on the set $\Omega_-$ as the argument for $\Omega_+$ is analogous:

As $u_0$ is constantly equal to $-1$ on $\Omega_-$ it holds for all $x, y \in \Omega_-$ and $z \in \Omega_-^c$
\begin{equation*}
  \sign(u_0(x) - u_0(y)) = 0 \quad \quad \text{ and } \quad \quad \sign(u_0(x) - u_0(z)) = -1.
\end{equation*}
This lets us compute
\begin{align*}
  \int_{\Omega_-} |H_h(x)|\,dx = \frac1{\sqrt{h}}\int_{\Omega_-} \int_{\Omega_-^c} G_h(x,y)\,dx.
\end{align*}
The right-hand side is exactly the thresholding energy which is known to converge to the perimeter for sets of finite perimeter, see for example~\cite[Lemma A.3]{EsedogluOtto}.

In order to prove~\eqref{eq:conv_to_H} we want to apply the uniform convergence~\eqref{eq:H_h_to_H}. To do this, fix $x \in \Omega_{:)}$. We need to ensure that $|\nabla u_0(x)|$ is bounded from below, $|\nabla^2 u_0|$ is bounded from above and $u_0 \in C^3(B_{\delta_0}(x))$. We first notice that $|\nabla u_0| = |\phi'(d)| |\nabla d| = \delta_0$; thus the gradient is bounded from below. Furthermore, the Hessian of $u_0$ is bounded from above on $\Omega_{:)}$ as $\nabla^2 u_0 = \nabla (\phi'(d) \nabla d) = \phi''(d) + \phi'(d) \nabla^2 d = \nabla^2 d$ and $\|\nabla^2 d\|_{L^\infty(B_{\delta_0}(x))} < \infty$.

As for $y \in \Omega_+$ (and similarly for $\Omega_-$) it always holds $\sign(u_0(x) - u_0(y)) = -1$ we can thus replace $u_0$ by any function that is smaller on $\Omega_+$ with out changing the $\sign$. But now we have for $\sqrt{h}R_h \leq \delta_0$ that
\begin{equation*}
   \frac{1}{\sqrt{h}} \int_{B_{\sqrt{h}R_h}(x)} G_h(x,y) \sign(u_0(x) - u_0(y)) \,dy = \frac{1}{\sqrt{h}} \int_{B_{\sqrt{h}R_h}(x)}  G_h(x,y) \sign(\tilde{u}_0(x) - \tilde{u}_0(y)) \,dy,
\end{equation*}
where $\tilde{u}_0 := 2 \phi(\frac{u_0}{2})$.
So $\tilde{u}_0$ is in $C^3(B_{\sqrt{h}R_h }(x))$ for $\sqrt{h}R_h \leq \delta_0$. Thus, we can apply~\eqref{eq:H_h_to_H} to obtain the, for all $x \in \Omega_{:)}$ uniform, convergence
\begin{equation}\label{eq:u_tilde}
  \lim_{h \downarrow 0}\frac{2}{\sqrt{h}} \int_{B_{\sqrt{h}R_h}(x)} G_h(x,y) \sign(u_0(x) - u_0(y)) \,dy = 2c_0\nabla \cdot \left(\frac{\nabla \tilde{u}_0(x)}{|\nabla \tilde{u}_0(x)|}\right) =  2c_0 \nabla \cdot \left(\frac{\nabla u_0(x)}{|\nabla u_0(x)|}\right).
\end{equation}

At last, we cut off the far field term at the radius $R_h = \sqrt{6|\log (h)|}$; this is possible as 
\begin{align}
  &\lim_{h \downarrow 0}\left|\frac{1}{\sqrt{h}} \int_{\R^d \setminus B_{\sqrt{h}R_h}(x)} G_1(x,y) \sign(u_0(x) - u_0(y)) \,dz\right| \nonumber\\
  &=
  \lim_{h \downarrow 0}\left|\frac{1}{\sqrt{h}} \int_{\R^d \setminus B_{R_h}(0)} G_1(0,z) \sign(u_0(x) - u_0(x + \sqrt{h} z)) \,dz\right|  \nonumber\\
  &\lesssim \lim_{h \downarrow 0}\frac{1}{\sqrt{h}} e^{-\frac{R_h^2}{8}}= \lim_{h \downarrow 0}\frac{h^{\frac{3}{4}}}{\sqrt{h}} = 0 \label{eq:cut_off_u_tilde}.
\end{align}
Combing~\eqref{eq:u_tilde} with~\eqref{eq:cut_off_u_tilde} yields~\eqref{eq:conv_to_H}.

It remains to prove~\eqref{eq:energy_bound_u_0}. To that aim we use that  $u_0$ is constant outside $B_L(0)$ to compute 
\begin{align*}
    E_h(u_0) = \frac{2}{\sqrt{h}} \int_{B_{L}(0)} \int_{\R^d} G_h(x,y) |u_0(x) - u_0(y)|\, dy \, dx.
\end{align*}
Next we use a coordinate transform $z = \frac{y-x}{\sqrt{h}}$ to obtain
\begin{equation*}
      \frac{2}{\sqrt{h}} \int_{B_{L}(0)} \int_{\R^d} G_h(x,y) |u_0(x) - u_0(y)|\, dy \, dx = \frac{2}{\sqrt{h}} \int_{B_L(0)} \int_{\R^d} G_1(0,z) |u_0(x) - u_0(x + \sqrt{h}z)|\, dz \, dx.
\end{equation*}
Finally, we use that $u_0$ is $1$-Lipschitz to get the control
\begin{multline*}
     \frac{2}{\sqrt{h}} \int_{B_L(0)} \int_{\R^d} G_1(0,z) |u_0(x) - u_0(x + \sqrt{h}z)|\, dz \, dx \\\leq 2 \int_{B_L(0)} \int_{\R^d} G_1(0,z) |z|\, dz \, dx  = 4c_0|B_L(0)| < \infty.
\end{multline*}
Combining the previous computations yields~\eqref{eq:energy_bound_u_0} and finishes this proof.
\end{proof}

\section*{Acknowledgments}
This project is funded by the Deutsche Forschungsgemeinschaft (DFG, German Research Foundation) under Germany's Excellence Strategy EXC 2181/1 - 390900948 (the Heidelberg STRUCTURES Excellence Cluster). The authors would like to thank the BANFF International Research Station for founding a research stay during the workshop ``Mathematical Analysis of Adversarial Machine Learning (25w5469)'' at Casa Matem\'atica Oaxaca (CMO) at which this project started.

\section*{AI-Declaration}
Most of the manuscript was written without the use of AI. The only usage was in parts of the proof of Lemma~\ref{lem:mon_time_derivative} (namely the argument~\eqref{eq:partial_t_contraction}--\eqref{eq:difference_quotient_second_order}) which is based on an interaction with ChatGPT Free Tier.
\printbibliography
\end{document}